\documentclass[11pt]{article}
\usepackage{mathrsfs}
\usepackage{bbm}
\usepackage{amsfonts}
\usepackage{amsmath,amssymb}
\usepackage{amsmath}
\usepackage{amssymb}
\usepackage{graphicx}
\usepackage{color}

\newtheorem{theorem}{Theorem}[section]
\newtheorem{corollary}[theorem]{Corollary}

\newtheorem{lemma}[theorem]{Lemma}
\newtheorem{proposition}[theorem]{Proposition}
\newtheorem{remark}[theorem]{Remark}

\newenvironment{proof}[1][Proof]{\noindent \textbf{#1.} }{\  $\Box$}
\numberwithin{equation}{section}

\begin{document}

\title{\textbf{ A global maximum principle for optimal control of general mean-field forward-backward stochastic systems with jumps}}
\author{Tao HAO\thanks{%
School of Statistics, Shandong University of Finance and Economics, Jinan 250014, P. R. China.
haotao2012@hotmail.com. Research supported by National Natural Science Foundation of China
(Grant Nos. 71671104,11871309,11801315,71803097), the Ministry of Education of Humanities and Social Science Project (Grant No. 16YJA910003),
Natural Science Foundation of Shandong Province (No. ZR2018QA001),
A Project of Shandong  Province Higher Educational Science and Technology Program (Grant Nos. J17KA162, J17KA163),
and Incubation Group Project of Financial Statistics and Risk Management of SDUFE.} \and %
Qingxin MENG\thanks{%
Qingxin  Meng is the corresponding author. Department of Mathematics, Huzhou University, Zhejiang 313000, P. R. China.
mqx@zjhu.edu.cn. Research supported by Natural Science Foundation of Zhejiang Province for Distinguished Young Scholar (Grant No. LR15A010001),
and the National Natural Science Foundation of China (Grant No. 11471079).}  }
\maketitle
\date{}

\begin{abstract}
In this paper we prove a  necessary condition of the optimal control problem for a class of general mean-field forward-backward stochastic systems with jumps in the case where the diffusion coefficients depend on control, the control set does not need to be convex, the coefficients of jump terms are independent
of control as well as the coefficients of mean-field backward stochastic differential equations depend on the joint law of $(X(t),Y(t))$.
Two new adjoint equations are brought in as well as several
new generic estimates of their solutions  are investigated for analysing the higher terms, especially, those involving the expectation which come from the derivatives
of the coefficients with respect to the measure. Utilizing these subtle estimates,
the second-order expansion of the cost functional, which is the key point to analyse the necessary condition,  is obtained, and whereafter the stochastic maximum principle.

\end{abstract}


\textbf{MSC-classification}: 93E20; 60H10\\\\


\noindent \section{{\protect \large {Introduction}}}

For some given measurable mappings $b,\sigma,\beta,f,\phi,$
we consider the general mean-field forward-backward stochastic differential equation (FBSDE):
\begin{equation}\label{equ 1.1}
\left\{
\begin{aligned}
dX^v(t)&=b(t,X^v(t),P_{X^v(t)},v(t))dt+\sigma(t,X^v(t),P_{X^v(t)},v(t))dW(t)\\
&\quad+\int_G\beta(t,X^v(t-),P_{X^v(t)},e)N_\lambda(de,dt),\ t\in[0,T],\\
 -dY^v(t)&=f(t,X^v(t),Y^v(t),Z^v(t),
 K^v(t,\cdot),P_{(X^v(t),Y^v(t))},v(t))dt-Z^v(t)dW(t)\\
 &\quad-\int_GK^v(t,e)N_\lambda(de,dt),\ t\in[0,T],\\
 X^v(0)&=x_0,\ Y^v(T)=\phi(X^v(T),P_{X^v(T)}),
\end{aligned}
\right.
\end{equation}
where $P_\eta:=P\circ\eta^{-1}$ denotes the probability measure induced by the random variable $\eta$.
Our control problem is to minimize a cost functional of the form $J(v(\cdot))=Y^v(0)$.
The purpose of this paper is to investigate the necessary condition of the above control problem
in the case where $\sigma$ depends on control and, moreover,
the action space is a general space, which means  it is needlessly convex.

The motivation comes on the one hand from the rapid  development of the theory of mean-field FBSDEs in recent years, in
particular, after the appearance of the notion of the derivative of a function with respect to a measure, refer to Lions \cite{Lions} or
Cardaliaguet \cite{Car}, on the other hand from the work of Hu \cite{Hu}, which solved the Peng's open problem \cite{Peng2} completely.

Stochastic maximum principle (SMP) is an important tool to study stochastic control problem. A lot of papers on this subject have been published.
The earliest works can be retrospected to Kushner \cite{Kushner} and Bismut \cite{Bismut}. The subsequent works
refer to Bensoussan \cite{Ben}, Haussmann \cite{Haussmann}, Peng \cite{Peng1}, Yong
and Zhou \cite{YZ}, Framstad, {\O}ksenal  and Sulem \cite{FOS}, and so on. All above works were investigated in
classical setting, not in a mean-field framework. As everyone knows, mean-field stochastic differential equations
(SDEs) (also named McKean-Vlasov equations)
have been considered by Kac \cite{Kac1} as long ago as 1956. However, due to the special structure, mean-field backward stochastic
differential equations (BSDEs) were not obtained by Buckdahn, Djehiche, Li and Peng \cite{BDLP}
until 2009 with a purely probabilistic approach. From then on, further progresses on mean-field BSDEs were provided by,
for example,
Buckdahn, Li, Peng \cite{BLP}, Buckdahn, Djehiche, Li \cite{BDL}. Especially, recently with the pioneer work of Lions \cite{Lions} to introduce the derivative
of a function defined on $\mathcal{P}_2(\mathbb{R}^d)$ with respect to a measure, the theory of general mean-field FBSDEs and related optimal control
problems or potential games stirred greatly the zeal of a large number of scholars. For instance, we refer
Lasry, Lions \cite{LL} for the theory of mean-field game,
refer Buckdahn, Li, Peng, Rainer \cite{BLPR}, Hao, Li \cite{HL},
Li \cite{Li}, Chassagneux, Crisan,  Delarue \cite{CCD} for the investigation of the relationship of the solutions of mean-field FBSDEs and corresponding PDEs, and
refer Carmona, Delarue \cite{CD1}, \cite{CD2}, Carmona, Delarue, Lachapelle \cite{CDL} for the description of the approximate Nash
equilibriums of symmetric games, i.e., the probability interpretation of mean-field game. Note that in \cite{CD1}, the authors
proved the existence of the approximate Nash equilibriums by making use of the tailor-made form of  SMP.
However, the assumptions on their SMP are  heavy, such as the Hamiltonian being strictly convex in control.
A natural question is whether the necessary condition of the  optimal control problem for general mean-field forward-backward stochastic systems (\ref{equ 1.1})  holds still true under some  slightly loose assumptions. In this paper, we confirm this declare.

Let us look at the structure of the equation (\ref{equ 1.1}) and show four main obstacles encountered in investigating the above
mean-field optimal control problem systemly:

a) The equation (\ref{equ 1.1}) is a general mean-field FBSDE. In fact, most of the existing works in mean-field framework can be summarized
as two cases:
$$ \mathrm{i})\ \mathbb E[\psi(t,x, X(t),v)]|_{x=X(t),v=v_t};\qquad  \mathrm{ii})\ \psi(t,\omega,X(t),
\mathbb E[X(t)],v_t).$$
 However, either of the above cases can be
put into the general type (\ref{equ 1.1}) by the definition of expectation and some simple transform, for example, for i)
$$\overline{\psi}(t,\omega,X(t),P_{X(t)},v_t)
:=\mathbb E[\psi(t,x,X(t),v)]|_{x=X(t),v=v_t}
=\int_{\mathbb{R}^n}\psi(t,x,y,v)P_{X(t)}
(dy)\Big|_{x=X(t),v=v_t}.$$
As we know, it is very difficult to analyse the second-order derivative of a function with respect to a measure.
Because even through a function is infinitely differentiable in usual sense, maybe it is not twice Fr\'{e}chet
differentiable. However, in this paper we want to study the optimal control problem in the case that $\sigma$ depends on control and
the control set is not convex, following the scheme of Peng \cite{Peng1} or Hu \cite{Hu}. So
the first obstacle is how to deal with the
second-order derivatives of the coefficients with respect to a measure and some new subtle estimates related them,
see Lemma \ref{le 3.2}, Corollary \ref{cor 3.3} and Corollary \ref{cor 3.6}, which are the building blocks to prove our SMP.

b) Closely related to our work is a paper by Buckdahn, Li and Ma \cite{BLM}, where the cost functional is of the form
$$J(v)=\mathbb E\bigg[\int_0^Tf(t,X^v(t),P_{X^v(t)},v_t)dt+\phi(X^v(T), P_{X^v(T)})\bigg].$$
However, in our case the coefficient $f$  depends not only
on $x$ and the law of $P_{X^v(t)}$, but also on $(y,z,k)$ and the joint law of $(X^v(t),Y^v(t))$, i.e., $P_{(X^v(t),Y^v(t))}$.
Two obstacles are met in this setting. The first one is that the power of the term $Y^1\delta\sigma(t)\mathbbm{1}_{E_\varepsilon}(t)$
in the variation of $z$ is $O(\varepsilon)$, but not $o(\varepsilon)$. For overcoming this difficulty, we construct an auxiliary
mean-field BSDE (\ref{equ 4.2}),
whose solution  satisfies very artful estimate (\ref{equ 4.3}). The second one is that
due to $f$ depending on the law of $(X^v(t),Y^v(t))$, the equation (\ref{equ 4.2}) is a mean-field type, which leads to that
the dual SDE is also mean-field type when applying dual method, see (\ref{equ 5.5}). But
it is not trivial to prove
 the solution
of (\ref{equ 5.5}) being larger than zero strictly, which is different to the classical case, see Hu \cite{Hu}.

c) It should be pointed out that although the dynamics involving jump term, the coefficient $\beta$ does not depend on control. Indeed,
if $\beta$ depends on control, for the solution of the first-order variational equation we only have the estimate: for $\ell\geq1$,
\begin{equation}\label{equ 1.3}
\mathbb E\bigg[\sup_{t\in[0,T]}
|X^{1,\varepsilon}(t)|^{2\ell}
\bigg]\leq L_\ell \varepsilon,
\end{equation}
but not
\begin{equation}\label{equ 1.4}
\mathbb E\bigg[\sup_{t\in[0,T]}
|X^{1,\varepsilon}(t)|^{2\ell}
\bigg]\leq L_\ell \varepsilon^{\ell},
\end{equation}
which is not sufficient to prove the SMP. This is the last obstacle we met.

There are two points, which should be lighten. Firstly, different to the classical case, not only the Taylor expansion of $X^\varepsilon$,
$X^\varepsilon=X^*+X^{1,\varepsilon}+X^{2,\varepsilon}+o(\varepsilon)$, but also the second-order expansion of cost functional $Y^\varepsilon$,
$Y^\varepsilon=Y^*+\breve{P}+Y^1+Y^2+o(\varepsilon)$ are needed, where the convergence of both of them are in $L^2(\Omega, C[0,T])$ sense.
To the best of our knowledge, the second expansion has not been seen in the existing literatures, in particular, $\breve{P}$ being a solution
of a linear mean-field BSDE.
Secondly, we establish some new and more generic estimates, see (\ref{equ 3.4-11}). Although the proof of
(\ref{equ 3.4-11}) follows the scheme of  Proposition 4.3 \cite{BLM}, the presence of the jump term makes the proof more technical.

The main result of this paper can be stated roughly as follows: Consider Hamiltonian
\begin{equation}\label{equ 1.5}
\begin{aligned}
&H(t,x,y,z,k,\nu,\mu,v;p,q,P)\\
&\quad=pb(t,x,\nu,v)+q\sigma(t,x,\nu,v)
+\frac{1}{2}P\Big(\sigma(t,x,\nu,v)-\sigma(t,X^*(t),P_{X^*(t)},u^*(t))\Big)^2\\
&\quad+f\Big(t,x,y,z+p\Big(\sigma(t,x,\nu,v)-\sigma(t,X^*(t),P_{X^*(t)},u^*(t))\Big), k,\mu, v\Big),
\end{aligned}
\end{equation}
$(t,x,y,z,k,\nu,\mu,v,p,q,P)\in[0,T]\times\mathbb{R}^3\times L^2(G,\mathscr{B}(G),\lambda)
\times \mathcal{P}_2(\mathbb{R})\times\mathcal{P}_2(\mathbb{R}^2)\times U \times\mathbb{R}^3.$\\
Let $u^*$ be the optimal control and $(X^*,Y^*,Z^*,K^*)$ be the optimal trajectory. By $(Y^i,Z^i,K^i),\ i=1,2$
we denote the solutions of the first- and second-order adjoint equations, respectively. Under some usual assumptions and
the additional assumptions
$$\widetilde{f}_{\mu_2}(t)=(\frac{\partial f}{\partial \mu})_2
(t,\widetilde{X}^*(t), \widetilde{Y}^*(t), \widetilde{Z}^*(t), \widetilde{K}^*(t,\cdot), P_{(X^*(t),Y^*(t))},\widetilde{u}^*(t);X^*(t),Y^*(t))>0,$$
$t\in[0,T],\ \widetilde{P}\otimes P$-a.s., and
$$f_k(t)=\frac{\partial f}{\partial k}(t,X^*(t), Y^*(t), Z^*(t), K^*(t,\cdot), P_{(X^*(t),Y^*(t))},u^*(t))>0,$$
$t\in[0,T], P$-a.s., we have
\begin{equation}\label{equ 1.6}
\begin{aligned}
&H(t,X^*(t),Y^*(t),Z^*(t),K^*(t,\cdot),P_{X^*(t)}, P_{(X^*(t),Y^*(t))},v;Y^1(t),Z^1(t),Y^2(t))\\
&\geq H(t,X^*(t),Y^*(t),Z^*(t),K^*(t,\cdot),P_{X^*(t)}, P_{(X^*(t),Y^*(t))},u^*(t);Y^1(t),Z^1(t),Y^2(t)),
\end{aligned}
\end{equation}
\indent \qquad\qquad\qquad\qquad\qquad\qquad\qquad\qquad\qquad\qquad\qquad\qquad\qquad
 $\forall v\in U$,\  a.e.,\ a.s.

This paper is arranged as follows. Section 2 recalls the notion of the derivative of a
function with respect to a measure and some notations. The formulation of the optimal control problem
is introduced in Section 3. The variational equations, the adjoint equations and the estimates of their
solutions are also given in this section. Section 4 is devoted to the introduction of the first important
conclusion of this paper---the second-order expansion of cost functional. The second important conclusion---SMP is proved in Section 5. In the last section some  necessary notations and the proof of the auxiliary
lemma are shown for closing our paper.

\section{{\protect \large {Preliminaries}}}

\subsection{Derivative of function $h:\mathcal{P}(\mathbb{R}^2)\rightarrow\mathbb{R}$}
Let $\mathcal{P}(\mathbb{R}^d)$ be the set of  all Borel probability measures on $\mathbb{R}^d$.
For $1\leq p< +\infty$,
let $\mathcal{P}_p(\mathbb{R}^d)$ be the subspace of $\mathcal{P}(\mathbb{R}^d)$ of probability measures having a finite moment of order $p$ over
$(\mathbb{R}^d, \mathscr{B}(\mathbb{R}^d))$, and moreover, we endow the space $\mathcal{P}_p(\mathbb{R}^d)$ with the $p$-Warsserstein metric: for
$\nu_1,\nu_2\in \mathcal{P}_p(\mathbb{R}^d)$,
$$W_p( \nu_1 ,  \nu_2 )=\inf\Big\{\Big[\int_{\mathbb{R}^{2d}}|x-y|^p\varrho(dx,dy) \Big]^{\frac{1}{p}}:  \varrho\in\mathcal{P}_p(\mathbb{R}^{2d}),\
\varrho(\cdot,\mathbb{R}^d)=\nu_1,\ \varrho(\mathbb{R}^d,\cdot)=\nu_2\Big\}. $$

We now recall the derivative of a function $h$ defined on $\mathcal{P}_2(\mathbb{R}^d)$ with respect to a measure, see Cardaliaguet \cite{Car}, or
Buckdahn, Li, Peng, Rainer \cite{BLPR} for more details.
We call the function $h:\mathcal{P}_2(\mathbb{R}^d)\rightarrow \mathbb{R}$ is differentiable in $\nu_0\in\mathcal{P}_2(\mathbb{R}^d)$, if
there exists a $\eta_0\in L^2(\mathcal{F};\mathbb{R}^d)$  with $\nu_0=P_{\eta_0}$, such that
the lifted function $\bar{h}:L^2(\mathcal{F};\mathbb{R}^d)\rightarrow \mathbb{R}$  defined by $\bar{h}(\eta):=h(P_\eta)$
is differentiable at $\eta_0$ in Fr\'{e}chet sense.
In other words, there exists a continuous linear functional $D\bar{h}(\eta_0): L^2(\mathcal{F};\mathbb{R}^d)\rightarrow \mathbb{R}$, such that
for $\eta\in L^2(\mathcal{F};\mathbb{R}^d)$,
\begin{equation}\label{equ 2.3}
\bar{h}(\eta_0+\eta)-\bar{h}(\eta_0)=D\bar{h}(\eta_0)(\eta)+o(||\eta||_{L^2}),
\end{equation}
with $||\eta||_{L^2}\rightarrow0$. From Riesz representation theorem and the argument of Cardaliaguet \cite{Car}, it follows that
there exists a Borel measurable function $g:\mathbb{R}^d\rightarrow\mathbb{R}^d$ depending only the law of $\eta_0$, but not the
random variable $\eta_0$ itself, such that (\ref{equ 2.3}) can read as

\begin{equation}\label{equ 2.4}
h(P_{\eta_0+\eta})-h(P_{\eta_0})=
<g(\eta_0), \eta>+o(||\eta||_{L^2}),
\end{equation}
where $<\cdot, \cdot>$ denotes the ``dual product" on $L^2(\mathcal{F};\mathbb{R}^d).$
From (\ref{equ 2.4}) we can define
$\partial_\nu h(P_{\eta_0};a):=g(a), a\in\mathbb{R}^d$, which is called
the derivative of $h$ at $P_{\eta_0}$. It should be pointed out that
the function $\partial_\nu h(P_{\eta_0};a)$ is only $P_{\eta_0}(da)$-a.e. uniquely determined.
In our case,  for simplicity we just consider those functions $h:\mathcal{P}_2(\mathbb{R}^d)\rightarrow\mathbb{R}$ being
differentiable in all elements of $\mathcal{P}_2(\mathbb{R}^d)$.

The following spaces have been introduced in \cite{BLPR}, \cite{HL}, \cite{Li}, \cite{BLM}. Here we borrow them. We denote

$\bullet$\ $C^{1,1}_b(\mathcal{P}_2(\mathbb{R}^d))$ to be all continuously differentiable function $h$ over $\mathcal{P}_2(\mathbb{R}^d)$
with Lipschitz-continuous bounded derivative, i.e., there exists a positive constant $C$ such that,\\
\mbox{}\qquad\qquad     (i)\ $|\partial_\nu h(\nu;a)|\leq C,\quad \forall a\in \mathbb{R}^d,\ \nu\in\mathcal{P}_2(\mathbb{R}^d)$;\\
\mbox{}\qquad\qquad     (ii)\ $|\partial_\nu h(\nu_1;a_1)-\partial_\nu h(\nu_2;a_2)|\leq C\Big(W_2(\nu_1,\nu_2)+|a_1-a_2|\Big),\  \nu_1,\nu_2\in\mathcal{P}_2(\mathbb{R}^d),\ a_1,a_2\in\mathbb{R}^d.$

$\bullet$\ $C_b^{2,1}(\mathcal{P}_2(\mathbb{R}^d))$ to be all measurable function $h\in C_b^{1,1}(\mathcal{P}_2(\mathbb{R}^d))$ satisfying:\\
\mbox{}\qquad\qquad   (i) for all $a\in\mathbb{R}^d$, $(\partial_\nu h)_\ell(\cdot;a)\in C_b^{1,1}(\mathcal{P}_2(\mathbb{R}^d)),\ \ell=1,2,\cdot,\cdot,\cdot,d$;\\
\mbox{}\qquad\qquad   (ii) for each $\nu\in\mathcal{P}_2(\mathbb{R}^d)$, $\partial_\nu h(\nu;\cdot)$ is differentiable;\\
\mbox{}\qquad\qquad   (iii) the second-order derivatives $\partial_a\partial_\nu h:\mathcal{P}_2(\mathbb{R}^d)\times
\mathbb{R}^d\rightarrow \mathbb{R}^d\otimes\mathbb{R}^d$ and
$\partial^2_\nu h(P_{\nu_0};a,b):=\partial_\nu (\partial_\nu h(\cdot;a))(P_{\nu_0};b):
\mathcal{P}_2(\mathbb{R}^d)\times\mathbb{R}^d\times\mathbb{R}^d\rightarrow\mathbb{R}^d\otimes\mathbb{R}^d$
are bounded and Lipschitz continuous.

%

\subsection{Function spaces}

Let $T$ be a fixed strictly positive real number and  $(\Omega,
\mathscr{F},\{\mathscr{F}_t\}_{0\leq t\leq T}, P)$ be a complete filtrated probability space on which a one-dimensional standard Brownian motion $\{W(t), 0\leq t\leq T\}$
is defined. Denote by $\mathscr{P}$
the $\mathscr{F}_t$-predictable $\sigma$-field on $[0, T]\times \Omega$ and by $\mathscr B(\Lambda)$
  the Borel $\sigma$-algebra of any topological space $\Lambda.$ Let $(G,\mathscr B (G), \lambda)$ be a measurable space with $\lambda(G)<\infty $ and $q: \Omega\times D_q \longrightarrow Z$ be an
$\mathscr F_t$-adapted stationary
Poisson point process with characteristic measure $\lambda$, where
 $D_q$  is a countable subset of $(0, \infty)$. Then the counting measure induced by $q$ is
$$
N_q((0,t]\times A):=\#\{s\in D_q; s\leq t, q(s)\in A\},~~~for~~~ t>0, A\in \mathscr B (G).
$$
Let
\begin{equation}\label{equ 2.1}
N_{\lambda,q}(de,dt):=N_q(de,dt)-\lambda(de)dt
\end{equation} be a compensated Poisson random martingale  measure which
is assumed to be independent of Brownian motion
$\{W(t), 0\leq t\leq T\}$. In what follows, when no confusion, we always omit the subscript $q$, and write (\ref{equ 2.1}) as
\begin{equation}\label{equ 2.2}
N_{\lambda}(de,dt)=N(de,dt)-\lambda(de)dt.
\end{equation}
 Assume $\mathbb{F}=\{\mathscr{F}_t\}_{0\leq t\leq T}$ is $P$-completed filtration generated by $\{W(t),
0\leq t\leq T\}$ and $\{\iint_{(0,t] \times A }$\\
 $N_{\lambda}(de,dt),0\leq t\leq T, A\in \mathscr B (G) \},$  and moreover, augmented by a $\sigma$-field $\mathscr{F}^o$ with the following property:\\
\indent(i)\ the Brownian motion $W$ and the Poisson random measure $N_\lambda$ are independent of $\mathscr{F}^o$;\\
\indent(ii)\ $\mathcal{P}_2(\mathbb{R}^d)=\{P_\eta,\ \eta\in L^2(\mathscr{F}^o;\mathbb{R}^d)\};$\\
\indent(iii)\ $\mathscr{F}^o$ contains the family of all the $P$-null subsets  $\mathcal{N}_P$.

The following several spaces are used frequently.

$\bullet$  By $L^p(\mathscr{F};\mathbb{R}^d)$ we denote the collection of $\mathbb{R}^d$-valued, $\mathscr{F}$-measurable random variables $\eta$
with $||\eta||_p:=\mathbb{E}[|\eta|^p]^{\frac{1}{p}}<+\infty$.

$\bullet$  By $\mathcal{S}^2_{\mathbb{F}}(0,T;\mathbb{R}^d)$ we denote the space of $\mathbb{R}^d$-valued, $\mathbb{F}$-predictable process $\varphi$
on $[0,T]$ with $\mathbb{E}[\sup_{0\leq s\leq T}|\varphi(t)|^2]<+\infty.$

$\bullet$   By $\mathcal{H}^2_{\mathbb{F}}(0,T;\mathbb{R}^d)$ we denote the space of all $\mathbb{R}^d$-valued, $\mathbb{F}$-adapted c\`{a}dl\`{a}g process $\varphi$ on $[0,T]$, such that  $\mathbb{E}[\int_0^T|\varphi(t)|^2dt]<+\infty.$

$\bullet$   By $\mathcal{K}^2_\lambda(0,T;\mathbb{R}^d)$ we denote
the space of  all $\mathbb{R}^d$-valued, $\mathscr{P}\times \mathscr{B}(G)$-measure process
$r$ on $[0,T]\times G$ satisfying
$E[\int_0^T\int_G|r(t,e)|^2\lambda(de)dt]<+\infty$.

Note that we denote $
 L^2(G,\mathscr{B}(G),\lambda;\mathbb{R}),
\mathcal{S}^2_{\mathbb{F}}(0,T;\mathbb{R}),\mathcal{H}^2_{\mathbb{F}}(0,T;\mathbb{R}),\mathcal{K}^2_\lambda(0,T;\mathbb{R})$
by
$ L^2(G,\mathscr{B}(G),\lambda),$ $\mathcal{S}^2_{\mathbb{F}}(0,T),\mathcal{H}^2_{\mathbb{F}}(0,T),\mathcal{K}^2_\lambda(0,T)$,
respectively, for short.

\section{Problem formulation, variational equations and adjoint equations}
\subsection{Problem formulation}
Let us first formulate the optimal control problem.
Let $U$ be a subset of $\mathbb{R}$. $v(\cdot):[0,T]\times \Omega\rightarrow U$ is called an
admissible control if $v(\cdot)$ is
$\mathcal{F}_t$-progressive measurable
process.
By $\mathcal{U}_{ad}$ we denote the set of all admissible controls. Let the mappings
$$
\begin{aligned}
&b:[0,T]\times \mathbb{R}\times \mathcal{P}_2(\mathbb{R})\times U\rightarrow \mathbb{R},\quad
\sigma:[0,T]\times \mathbb{R}\times \mathcal{P}_2(\mathbb{R})\times U\rightarrow \mathbb{R},\\
&\beta:[0,T]\times \mathbb{R}\times \mathcal{P}_2(\mathbb{R})\times G\rightarrow \mathbb{R},\quad
\phi: \mathbb{R}\times \mathcal{P}_2(\mathbb{R})\rightarrow \mathbb{R},\\
&f:[0,T]\times \mathbb{R}\times \mathbb{R}\times \mathbb{R}\times L^2(G,\mathscr{B}(G),\lambda)\times\mathcal{P}_2(\mathbb{R}^2)\times U\rightarrow \mathbb{R},
\end{aligned}
$$
satisfy: \\
$\underline{\text{Assumption}\ (\mathbf{A3.1})}$   the measurable mappings $b,\sigma,f, \phi$ are bounded,
 and for each $e\in G$, $\beta$ is bounded by $C(1\wedge|e|)$ with $C$ independent of $e\in G$;\\
$\underline{\text{Assumption}\ (\mathbf{A3.2})}$ for $t\in[0,T],u\in U, e\in G$, $(b,\sigma)(t,\cdot,\cdot,u)\in C_b^{1,1}(\mathbb{R}\times \mathcal{P}_2(\mathbb{R}))$,
$\beta(t,\cdot,\cdot,e)\in C_b^{1,1}(\mathbb{R}\times \mathcal{P}_2(\mathbb{R}))$,
$ f(t,\cdot,\cdot,\cdot,\cdot,\cdot,u)\in C_b^{1,1}(\mathbb{R}^3\times L^2(G,\mathscr{B}(G),\lambda)\times \mathcal{P}_2(\mathbb{R}^2))$,
$\phi(\cdot,\cdot)\in C_b^{1,1}(\mathbb{R}\times \mathcal{P}(\mathbb{R}))$, i.e.,\\
\indent (i) for $(t,x,y,z,k,u,e)\in[0,T]\times \mathbb{R}^3\times L^2(G,\mathscr{B}(G),\lambda)\times U\times G$,
$(b,\sigma)(t,x,\cdot,u)\in C_b^{1,1}(\mathcal{P}_2(\mathbb{R})),$
$\beta(t,x,\cdot,e)\in C_b^{1,1}(\mathcal{P}_2(\mathbb{R}))$, $f(t,x,y,z,k,\cdot,u)\in C_b^{1,1}(\mathcal{P}_2(\mathbb{R}^2)),$
$\phi(x,\cdot)\in C_b^{1,1}(\mathcal{P}_2(\mathbb{R}))$;\\
\indent (ii) for $(t,\nu,\mu,u,e)\in[0,T]\times\mathcal{P}_2(\mathbb{R})\times\mathcal{P}_2(\mathbb{R}^2)\times U\times G$,
$(b,\sigma)(t,\cdot,\nu,u)\in C_b^1(\mathbb{R}),$ $\beta(t,\cdot,\nu,e)\in C_b^1(\mathbb{R})$,
$f(t,\cdot,\cdot,\cdot,\cdot,\mu,u)\in C_b^1(\mathbb{R}^3\times L^2(G,\mathscr{B}(G),\lambda))$, $\phi(\cdot,\nu)\in C_b^1(\mathbb{R})$;\\
\indent (iii)  all the first-order derivatives $\partial_\ell\psi$,\  $\psi=b,\sigma,f,\phi,\ \ell=x,y,z,\nu,\mu$
are bounded and Lipschitz continuous with the constant independent of $u\in U$; for each
$e\in G$, $\partial_x \beta,\ \partial_\nu\beta$ are bounded by $C(1\wedge|e|)$ and Lipschitz continuous with
the constant  $C$ independent of $u\in U$ and $e\in G$;\\
and, furthermore,  \\
$\underline{\text{Assumption}\ (\mathbf{A3.3})}$  let $b,\sigma,\beta, f,\phi$ satisfy Assumptions (A3.1)-(A3.2), and, meanwhile,
for $t\in[0,T],u\in U, e\in G$, $(b,\sigma)(t,\cdot,\cdot,u)\in C_b^{2,1}(\mathbb{R}\times \mathcal{P}_2(\mathbb{R}))$,
$\beta(t,\cdot,\cdot,e)\in C_b^{2,1}(\mathbb{R}\times \mathcal{P}_2(\mathbb{R}))$,
$ f(t,\cdot,\cdot,\cdot,\cdot,\cdot,u)\in C_b^{2,1}(\mathbb{R}^3\times L^2(G,\mathscr{B}(G),\lambda)\times \mathcal{P}_2(\mathbb{R}^2))$,
$\phi(\cdot,\cdot)\in C_b^{2,1}(\mathbb{R}\times \mathcal{P}(\mathbb{R}))$, i.e., the derivatives of $b,\sigma,\beta,f,\phi$
enjoy the following properties:\\
\indent(i) for  $(t,u,e)\in[0,T]\times U\times G$, $(\partial_x b, \partial_x \sigma)(t,\cdot,\cdot,u)\in C_b^{1,1}(\mathbb{R}\times \mathcal{P}_2(\mathbb{R})),$
$\partial_x \beta(t,\cdot,\cdot,e)\in C_b^{1,1}(\mathbb{R}\times \mathcal{P}_2(\mathbb{R}))$,
$\partial_\ell f(t,\cdot,\cdot,\cdot,\cdot,\cdot,u)\in C_b^{1,1}(\mathbb{R}^3\times L^2(G,\mathscr{B}(G),\lambda)\times \mathcal{P}_2(\mathbb{R}^2))$,
$\ell=x,y,z,k$, $\partial_x\phi(\cdot,\cdot)\in C_b^{1,1}(\mathbb{R}\times \mathcal{P}(\mathbb{R}))$;\\
\indent(ii) for  $(t,u,e)\in[0,T]\times U\times G$,
$(\partial_\nu b, \partial_\nu \sigma)(t,\cdot,\cdot,u;\cdot)\in C_b^{1,1}(\mathbb{R}\times \mathcal{P}_2(\mathbb{R})\times \mathbb{R}),$
$\partial_\nu \beta(t,\cdot,\cdot,e;\cdot)\in C_b^{1,1}(\mathbb{R}\times \mathcal{P}_2(\mathbb{R})\times\mathbb{R})$,
$(\partial_{\mu} f)_j(t,\cdot,\cdot,\cdot,\cdot,\cdot,u;\cdot,\cdot)\in C_b^{1,1}(\mathbb{R}^3\times L^2(G,\mathscr{B}(G),\lambda)\times \mathcal{P}_2(\mathbb{R}^2)\times\mathbb{R}^2),\ j=1,2$,
$\partial_\nu\phi(\cdot,\cdot;\cdot)\in C_b^{1,1}(\mathbb{R}\times \mathcal{P}(\mathbb{R})\times\mathbb{R})$;\\
\indent(iii)  all the second-order derivatives of $b,\sigma,f,\phi$ are bounded and Lipschitz continuous with
the Lipschitz constants independent of $u\in U$; for each $e\in G$, all the second-order derivatives of $\beta$
are bounded by $C(1\wedge|e|)$,  and Lipschitz continuous with the constant $C$ independent of $(e,u)\in G\times U$.\\

For $v(\cdot)\in \mathcal{U}_{ad}$,  under the Assumptions (A3.1)-(A3.2), the equation (\ref{equ 1.1}) possesses a unique solution
$(X^v,Y^v,Z^v,K^v).$

The target of the optimal control problem consists in minimizing $J(v(\cdot))=Y^{v}(0)$ over $\mathcal{U}_{ad}.$
In other words, whether there exists a $u^*(\cdot)\in\mathcal{U}_{ad}$ such that
\begin{equation}\label{equ 3.1-1}
J(u^*(\cdot))=\inf_{v\in \mathcal{U}_{ad}}J(v(\cdot)).
\end{equation}
The main purpose of this paper is to study the  necessary condition of the optimal control problem (\ref{equ 1.1}) and (\ref{equ 3.1-1}).

\begin{remark}
Throughout this paper, we  set $\rho:(0,+\infty)\rightarrow(0,+\infty)$ is  a function with the property
$\rho(\varepsilon)\rightarrow0,$   as  $\varepsilon\rightarrow0$, and $C$ is a positive constant, both of which
maybe change from one appearance to another.
\end{remark}

 \subsection{Variational equations}
This subsection is devoted to the introduction of the first- and  second-order variational equations,
as well as some estimates of their solutions.

Now let  $u^*(\cdot)\in \mathcal{U}_{ad}$ be an optimal control, and by $(X^*(\cdot),Y^*(\cdot),Z^*(\cdot),
K^*(\cdot))=(X^{u^*}(\cdot),Y^{u^*}(\cdot),
\\Z^{u^*}(\cdot),K^{u^*}(\cdot))$,
the solution of (\ref{equ 1.1}) with $u^*(\cdot)$ instead of $v(\cdot)$, we denote the optimal state process.
It is clear from the definition of a function with respect to a measure, that when studying the first- and second-order derivatives of
coefficients with respect to a measure, some auxiliary probability spaces are needed. Hence we would like to introduce first
an intermediate probability space and the stochastic processes defined on it as a representative. The other probability spaces and
corresponding stochastic processes can be understood in the same sense.  For this end,
let  $(\overline{\Omega} , \overline{\mathscr{F}},\overline{P})$ be
an intermediate complete probability space, which is independent of $(\Omega, \mathscr{F},P)$.
The pair $(\overline{W},\overline {N_\lambda})$ defined on space $(\overline{\Omega}, \overline{\mathscr{F}},\overline{P})$ is an independent copy of
$(W,N_\lambda)$, i.e., $(\overline{W},\overline{N_\lambda})$ under $\overline{P}$ has the same law as $(W,N_\lambda)$ under $P$.
By $\overline{X}^v(\cdot)$ we denote the corresponding state trajectory but driven by $(\overline{W},\overline{N_\lambda})$
instead of $(W,N_\lambda)$ in (\ref{equ 1.1}).
$\overline{\mathbb E}[\cdot]$ only acts on the random variables or/and the stochastic
processes with ``bar".
$(\widetilde{\Omega},
\widetilde{\mathscr{F}},
\widetilde{P},\widetilde{X}(t),
\widetilde{\mathbb E}[\cdot])$
and
$(\widehat{\Omega},\widehat{\mathscr{F}},
\widehat{P},\widehat{X}(t),\widehat{\mathbb E}[\cdot])$
can be understood in the same meaning. Note that $
(\Omega, \mathscr{F},P),
(\widehat{\Omega}, \widehat{\mathscr{F}},\widehat{P})$, $(\widetilde{\Omega},\widetilde{\mathscr{F}},\widetilde{P})$ and $(\overline{\Omega},\overline{\mathscr{F}},\overline{P})$
are also independent.

Let $v(\cdot)$ be any given
admissible control. For $\phi=b, \sigma, b_x, \sigma_x$  and $\psi=b, \sigma$,  define
$$
\begin{aligned}
&\delta\phi(t):=\phi(t,X^*(t),P_{X^*(t)},v(t))-\phi(t,X^*(t),P_{X^*(t)},u^*(t)),\\
&(\psi_x,\psi_{xx})(t):=(\frac{\partial\psi}{\partial x} ,\frac{\partial^2\psi}{\partial x^2})(t,X^*(t),P_{X^*(t)},u^*(t)),\\
&(\psi_\nu,\psi_{\nu a})(t;\widetilde{X}^*(t)):=(\frac{\partial\psi}{\partial \nu} ,\frac{\partial^2\psi}{\partial \nu \partial a})(t,X^*(t),P_{X^*(t)},u^*(t);\widetilde{X}^*(t)),\\
&(\widetilde{\psi}_\nu,\widetilde{\psi}_{\nu a})(t)=(\frac{\partial\psi}{\partial \nu} ,\frac{\partial^2\psi}{\partial \nu \partial a})(t,\widetilde{X}^*(t), P_{ X^*(t)},\widetilde{u}^*(t);X^*(t)),\\
&\psi_{\nu \nu}(t;\widehat{\widetilde{X}}^*(t)):=\frac{\partial^2\psi}{\partial \nu^2}(t,X^*(t),P_{X^*(t)},u^*(t);
\widetilde{X}^*(t), \widehat{X}^*(t)).\\
\end{aligned}
$$

Let $\varepsilon>0$, and $E_\varepsilon\subset[0,T]$ be a Borel set with Borel measure $|E_\varepsilon|=\varepsilon$.
For any $v(\cdot)\in\mathcal{U}_{ad}$, we consider the ``spike variation" of the optimal control $u^*(\cdot)$:
$u^\varepsilon(t):=u^*(t)\mathbbm{1}_{(E_\varepsilon)^c}+v(t)\mathbbm{1}_{E_\varepsilon}$, and
let $(X^\varepsilon, Y^\varepsilon,Z^\varepsilon,K^\varepsilon):
=(X^{u^\varepsilon},Y^{u^\varepsilon},
Z^{u^\varepsilon},K^{u^\varepsilon})$
 be the solution of (\ref{equ 1.1}) under the control $u^\varepsilon(\cdot).$
Inspired by Peng \cite{Peng1}, when the control is involved in the diffusion term and  the control domain is  not convex,
for each $\varepsilon>0$, one can find two processes $X^{1,\varepsilon}$ and $X^{2,\varepsilon}$, such that $X^\varepsilon-X^*-X^{1,\varepsilon}=O(\varepsilon)$,
and $X^\varepsilon-X^*-X^{1,\varepsilon}-X^{2,\varepsilon}=o(\varepsilon)$,
where  the convergence are both in $L^2(\Omega, C[0,T])$  sense.
In our case it is easy to check that the first- and second-order variational equations $X^{1,\varepsilon}$ and  $X^{2,\varepsilon}$ satisfy
\begin{equation}\label{equ 3.2}
  \left\{
 \begin{aligned}
dX^{1,\varepsilon}(t)&=\Big\{ b_x(t)X^{1,\varepsilon}(t)+\widetilde{\mathbb E}[b_\nu(t;\widetilde{X}^*(t))\widetilde{X}^{1,\varepsilon}(t)]+\delta b(t)\mathbbm{1}_{E_\varepsilon}(t)\Big\}dt\\
        &\quad+\Big\{ \sigma_x(t)X^{1,\varepsilon}(t)+\widetilde{\mathbb E}[\sigma_\nu(t;\widetilde{X}^*(t))\widetilde{X}^{1,\varepsilon}(t)]+\delta \sigma(t)\mathbbm{1}_{E_\varepsilon}(t)\Big\}dW(t)\\
        &\quad+\int_G\Big\{  \beta^-_x(t,e)X^{1,\varepsilon}(t-)+\widetilde{\mathbb E}[\beta^-_\nu(t,e;\widetilde{X}^*(t))\widetilde{X}^{1,\varepsilon}(t-)]\Big\}N_\lambda(de,dt),\ t\in[0,T],\\
X^{1,\varepsilon}(0)&=0,
\end{aligned}
\right.
\end{equation}
and
\begin{equation}\label{equ 3.3}
  \left\{
\begin{aligned}
dX^{2,\varepsilon}(t)&=\Big\{ b_x(t)X^{2,\varepsilon}(t)+\widetilde{\mathbb E}[b_\nu(t;\widetilde{X}^*(t))\widetilde{X}^{2,\varepsilon}(t)]
       +\frac{1}{2}\Big(b_{xx}(t)(X^{1,\varepsilon}(t))^2+\widetilde{\mathbb E}[b_{\nu a}(t;\widetilde{X}^*(t))\\
        &\qquad\ (\widetilde{X}^{1,\varepsilon}(t))^2] \Big)+\Big(\delta b_x(t)X^{1,\varepsilon}(t)+\widetilde{\mathbb E}[\delta b_\nu(t;\widetilde{X}^*(t))\widetilde{X}^{1,\varepsilon}(t)]\Big)\mathbbm{1}_{E_\varepsilon}(t)\Big\}dt\\
        &\quad+\Big\{ \sigma_x(t)X^{2,\varepsilon}(t)+\widetilde{\mathbb E}[\sigma_\nu(t;\widetilde{X}^*(t))\widetilde{X}^{2,\varepsilon}(t)]
         +\frac{1}{2}\Big(\sigma_{xx}(t)(X^{1,\varepsilon}(t))^2+\widetilde{\mathbb E}[\sigma_{\nu a}(t;\widetilde{X}^*(t)) \\
         &\qquad\ (\widetilde{X}^{1,\varepsilon}(t))^2] \Big)+\Big(\delta \sigma_x(t)X^{1,\varepsilon}(t)+\widetilde{\mathbb E}[\delta \sigma_\nu(t;\widetilde{X}^*(t))\widetilde{X}^{1,\varepsilon}(t)]\Big)\mathbbm{1}_{E_\varepsilon}(t)\Big\}dW(t)\\
         &\quad+\int_G\Big\{ \beta^-_x(t,e)X^{2,\varepsilon}(t-)+\widetilde{\mathbb E}[\beta^-_\nu(t,e;\widetilde{X}^*(t))\widetilde{X}^{2,\varepsilon}(t-)]\\
         &\quad+\frac{1}{2}\Big(\beta^-_{xx}(t,e)(X^{1,\varepsilon}(t-))^2+\widetilde{\mathbb E}[\beta_{\nu a}^-(t,e;\widetilde{X}^*(t))(\widetilde{X}^{1,\varepsilon}(t-))^2] \Big)\Big\}N_\lambda(de,dt),
         \quad t\in[0,T],\\
X^{2,\varepsilon}(0)&=0.
 \end{aligned}
 \right.
\end{equation}
Here\ \   $(\beta_x^-,\beta_{xx}^-)(t,e)=(\frac{\partial\beta}{\partial x},\frac{\partial^2\beta}{\partial x^2})(t,X^*(t-),P_{X^*(t)},e),$\ \
$(\beta_\nu^-,\beta_{\nu a}^-)(t,e;\widetilde{X}^*(t))=(\frac{\partial\beta}{\partial \nu},\frac{\partial^2\beta}{\partial \nu   \partial a}) (t,X^*(t-),$ $P_{X^*(t)},e;\widetilde{X}^*(t-)).
$

Obviously, under the Assumptions (A3.1)-(A3.3), the equation (\ref{equ 3.2}) and the equation (\ref{equ 3.3}) have unique solutions $\{X^{1,\varepsilon}(t)\}_{t\in[0,T]}$ and
$\{X^{2,\varepsilon}(t)\}_{t\in[0,T]}$. Moreover, their solutions satisfy the following estimates:
\begin{proposition}\label{pro 3.1}
Let the Assumptions (A3.1)-(A3.3)  hold true. For $\ell\geq1$, there exists a  constant $L_\ell>0$  depending only on $\ell$ such that
\begin{equation}\label{equ 3.4}
\begin{aligned}
&\mathrm{i)}\quad  \mathbb E\bigg[\sup_{t\in[0,T]}|X^{1,\varepsilon}
(t)|^{2\ell}\bigg]\leq L_\ell \varepsilon^\ell,
\quad  \mathbb E\bigg[\sup_{t\in[0,T]}
|X^{2,\varepsilon}(t)|^{2\ell}
\bigg]\leq L_\ell \varepsilon^{2\ell}; \\
&\mathrm{ii)}\quad  \mathbb E\bigg[\sup_{t\in[0,T]}|X^\varepsilon(t)-X^*(t)|^{2\ell}]\leq L_\ell \varepsilon^{\ell};\\
&\mathrm{iii)}\quad  \mathbb E\bigg[\sup_{t\in[0,T]}|X^\varepsilon(t)
-X^*(t)-X^{1,\varepsilon}(t)|^{2\ell}
\bigg]\leq L_\ell \varepsilon^{2\ell}.\\
\end{aligned}
\end{equation}
\end{proposition}
The proof is similar to Proposition 4.2 in \cite{BLM}. Hence, we omit it.

An extra assumption is the need to prove the following lemma.\\
\noindent $\underline{\text{Assumption}\ (\mathbf{A3.4})}$ Let
 $1+\beta_x(t,e)\geq\delta $,  $(t,e)\in[0,T]\times G,$ where $\delta $
 is some given positive constant.

\begin{lemma}\label{le 3.2}
Let the Assumptions (A3.1), (A3.2) and (A3.4) hold true and
let $(\overline{\Omega},\overline{\mathcal{F}},\overline{P})$
be an intermediate probability space and independent of space of space $(\Omega, \mathcal{F},P)$, and
 let $(\overline{\psi}_3(t,e))_{(t,e)\in[0,T]\times G}$,\\ $(\overline{\psi}_1(t))_{t\in[0,T]}$ be two progressively measurable
  stochastic
 processes defined on the product space $(\Omega\times\overline{\Omega},\mathcal{F}\times\overline{\mathcal{F}},P\otimes \overline{P})$
and $(\overline{\psi}_2(t))_{t\in[0,T]}$ be a progressively measurable stochastic process defined on
the space $(\overline{\Omega},\overline{\mathcal{F}},\overline{P})$. Moreover, assume
$(\overline{\psi}_i(t))_{t\in[0,T]},\ i=1,2,3$
satisfies the
following properties:

$$
\begin{aligned}
&\mathrm{a)}\ \text{for}\ e\in G,\ t\in[0,T],\ \    |\overline{\psi}_1(t)|\leq C, \ \ |\overline{\psi}_3(t,e)|\leq C(1\wedge|e|),\ P\otimes \overline{P}\text{-a.s.},\\
&\mathrm{b)}\ \text{for}\ \ell\geq1, \ \ \overline{\mathbb E}\Big[\sup_{t\in[0,T]}|\overline{\psi}_2(t)|^
{2\ell}\Big]\leq C_\ell.
\end{aligned}
$$
Then
\begin{equation}\label{equ 3.4-11}
\begin{aligned}
&\mathrm{i)}\  \mathbb E\bigg[\int_0^T\Big|\overline{
\mathbb E} [\overline{\psi}_1(t)
\overline{\psi}_2(t)\overline{X}^{1,
\varepsilon}(t) ]\Big|^4dt\bigg]\leq \varepsilon^2\rho(\varepsilon),\\
&\mathrm{ii)}\ \mathbb E\bigg[\int_0^T\int_G\Big|\overline{\mathbb E} [\overline{\psi}_3(t,e)
\overline{\psi}_2(t)\overline{X}^{1,\varepsilon}(t) ]\Big|^4\lambda(de)dt\bigg]\leq \varepsilon^2\rho(\varepsilon).
\end{aligned}
\end{equation}
\end{lemma}

\begin{proof}
Under the Assumptions (A3.1), (A3.2) and (A3.4), the proof of i) follows that of Proposition 4.3 \cite{BLM}. Hence,
we mainly estimate ii). The proof of ii) is an adaptation of that for Proposition 4.3 \cite{BLM}. Let us state it in detail.
Denote
\begin{equation}\label{equ 3.5}
\begin{aligned}
S(t)&=\int_0^t\Big(-b_x(s)+\frac{1}{2}|\sigma_x(s)|^2+\int_G(\beta_x(s,e)-\mathrm{ln}(1+\beta_x(s,e)))\lambda(de)\Big) ds\\
    &\quad-\int_0^t\sigma_x(s)dW(s)-\int_0^t\int_G \mathrm{ln}(1+\beta_x(s,e))N_\lambda(de,ds),
\end{aligned}
\end{equation}
and consider $m(t)=e^{S(t)}$.  Obviously,  $(m(s))_{s\in[0,T]}$ satisfies\\
\begin{equation}\label{equ 3.6}
\left\{
\begin{aligned}
dm(t)&=m(t)\Big(-b_x(t)+|\sigma_x(t)|^2
+\int_G\frac{|\beta_x(t,e)|^2}{1+\beta_x(t,e)}\lambda(de)\Big)dt-m(t)\sigma_x(t)dW(t)\\
     &\quad-
     \int_Gm(t-)\frac{\beta^-_x(t,e)}
     {1+\beta^-_x(t,e)}N_\lambda(de,dt),\ t\in[0,T],\\
m(0)&=1.
\end{aligned}
\right.
\end{equation}

Let $n(t)=m(t)^{-1}=e^{-S(t)}$.  Due to $\beta_x(t,e)\geq\delta-1>-1$,  $(t,e)\in[0,T]\times G$, the boundness of $b_x, \sigma_x,\beta_x$ implies that,
for $\ell\geq1$,
\begin{equation}\label{equ 3.7}
\mathbb E\Big[\sup_{t\in[0,T]}\big(|n(t)|^\ell+|m(t)|^\ell\big)\Big]\leq C_\ell,
\end{equation}
where $C_\ell$ is a positive constant only depending on $\ell$.

On the other hand,  it follows from It\^{o}'s formula for semi-martingale with jumps (see Theorem 93 \cite{Situ}) that
\begin{equation}\label{equ 3.8}
\begin{aligned}
dX^{1,\varepsilon}(t)m(t)&=m(t)\Big(\widetilde{\mathbb E}[\sigma_\nu(t,\widetilde{X}^*(t))\widetilde{X}^{1,\varepsilon}(t)]+\delta \sigma(t)\mathbbm{1}_{E_\varepsilon}(t)\Big)dW(t)\\
&\quad+\int_Gm(t-)\frac{1}{1+\beta^-_x(t,e)}\widetilde{\mathbb E}[\beta_\nu^-(t,e;\widetilde{X}^*(t))\widetilde{X}^{1,\varepsilon}(t-)]N_\lambda(dt,de)\\
&\quad+\Big\{
m(t)\big(\widetilde{\mathbb E}[b_\nu(t;\widetilde{X}^*(t))\widetilde{X}^{1,\varepsilon}(t)]+\delta b(t)\mathbbm{1}_{E_\varepsilon}(t)\big)\\
&\quad-m(t)\sigma_x(t)\big(
\widetilde{\mathbb E}[\sigma_\nu(t;\widetilde{X}^*(t))\widetilde{X}^{1,\varepsilon}(t)]+\delta \sigma(t)\mathbbm{1}_{E_\varepsilon}(t)
\big) \\
&\quad-\int_Gm(t)\frac{\beta_x(t,e)}{1+\beta_x(t,e)}
\widetilde{\mathbb E}[\beta_\nu(t,e;\widetilde{X}^*(t))\widetilde{X}^{1,\varepsilon}(t)]\lambda(de)
\Big\}dt.
\end{aligned}
\end{equation}
Hence,
\begin{equation}\label{equ 3.9}
\begin{aligned}
X^{1,\varepsilon}(t)=n(t)\Theta_1^\varepsilon(t)+n(t)\Theta_2^\varepsilon(t)+\Theta_3^\varepsilon(t),
\end{aligned}
\end{equation}
where
\begin{equation}\label{equ 3.10}
\begin{aligned}
\Theta_1^\varepsilon(t)&:=\int_0^tm(s)\Big(\widetilde{\mathbb E}[\sigma_\nu(s;\widetilde{X}^*(s))\widetilde{X}^{1,\varepsilon}(s)]+\delta \sigma(s)\mathbbm{1}_{E_\varepsilon}(s)\Big)dW(s),\\
\Theta_2^\varepsilon(t)&:=\int_0^t\int_Gm(s-)\frac{1}{1+\beta^-_x(s,e)}\widetilde{\mathbb E}[\beta_\nu^-(s,e;\widetilde{X}^*(s))
\widetilde{X}^{1,\varepsilon}(s-)]N_\lambda(ds,de),\\
\Theta_3^\varepsilon(t)&:=n(t)\int_0^t\Big\{
m(s)\big(\widetilde{\mathbb E}[b_\nu(s;\widetilde{X}^*(s))\widetilde{X}^{1,\varepsilon}(s)]+\delta b(s)\mathbbm{1}_{E_\varepsilon}(s)\big)\\
&\qquad\qquad\quad-m(s)\sigma_x(s)\big(
\widetilde{\mathbb E}[\sigma_\nu(s;\widetilde{X}^*(s))\widetilde{X}^{1,\varepsilon}(s)]+\delta \sigma(s)\mathbbm{1}_{E_\varepsilon}(s)
\big) \\
&\qquad\qquad\quad-\int_Gm(s)\frac{\beta_x(s,e)}{1+\beta_x(s,e)}
\widetilde{\mathbb E}[\beta_\nu(s,e;\widetilde{X}^*(s))\widetilde{X}^{1,\varepsilon}(s)]\lambda(de)
\Big\}ds.
\end{aligned}
\end{equation}
We are now ready to calculate $\overline{\mathbb E}[\overline{\psi}_3(t,e')\overline{\psi}_2(t)\overline{X}^{1,\varepsilon}(t)]$ with the help of
(\ref{equ 3.9}) and (\ref{equ 3.10}).\\
For each given $e'\in G$,
\begin{equation}\label{equ 3.11}
\begin{aligned}
\overline{\mathbb E}[\overline{\psi}_3(t,e')\overline{\psi}_2(t)\overline{X}^{1,\varepsilon}(t)]&=
\overline{\mathbb E}[\overline{\psi}_3(t,e')
\overline{\psi}_2(t)\overline{n}(t)\overline{\Theta}_1^\varepsilon(t)]
+\overline{\mathbb E}[\overline{\psi}_3(t,e')
\overline{\psi}_2(t)\overline{n}(t)
\overline{\Theta}_2^\varepsilon(t)]\\
&\quad+\overline{\mathbb E}[\overline{\psi}_3(t,e')
\overline{\psi}_2(t)\overline{\Theta}_3^
\varepsilon(t)]
:=\Xi_1^\varepsilon(t,e')+\Xi_2^\varepsilon(t,e')
+\Xi_3^\varepsilon(t,e').
\end{aligned}
\end{equation}
Since, for $\ell\geq1,$
\begin{equation}\label{equ 3.12}
\begin{aligned}
&\overline{\mathbb {E}}\bigg[\sup_{t\in[0,T]}|\overline{\psi}_3(t,e')\overline{\psi}_2(t)\overline{n}(t)|^\ell]\leq C\mathbb \mathbb E\bigg[\sup_{t\in[0,T]}|\overline{\psi}_2(t)\overline{n}(t)|^\ell\bigg]\\
&\leq C\Big\{\mathbb E[\sup_{t\in[0,T]}|\overline{\psi}_2(t)|^{2\ell}]\Big\}^{\frac{1}{2}}\cdot
\Big\{\mathbb E[\sup_{t\in[0,T]}|\overline{n}(t)|^{2\ell}]\Big\}^{\frac{1}{2}}\leq C_\ell,
\end{aligned}
\end{equation}
where $C_\ell$ does not depends on $e'$ because of $|\psi_3(t,e')|\leq C(1\wedge|e'|)\leq C,$
and observe that $\mathbb{F}=\mathbb{F}^W\bigvee \mathbb{F}^N$, according to the martingale
representation theorem for jump process (see \cite{TL}), we have, for each  $t\in[0,T],\ e'\in G$, there
exists a unique pair
$(\overline{\theta}_{\cdot,t,e'},\overline{\gamma}_{\cdot,t,e'})\in \mathcal{H}^2_\mathbb{\overline{F}}(0,t)\times \overline{K}_\lambda^2(0,t),$
such that $\overline{P}$-a.s.,
\begin{equation}\label{equ 3.13}
\begin{aligned}
\overline{\psi}_3(t,e')\overline{\psi}_2(t)\overline{n}(t)=
\mathbb E[\overline{\psi}_3(t,e')\overline{\psi}_2(t)\overline{n}(t)]
+\int_0^t\overline{\theta}_{s,t,e'}d\overline{W}(s)
+\int_0^t\overline{\gamma}_{s,t,e'}(e)\overline{N}_\lambda(de,ds).
\end{aligned}
\end{equation}

We argue that, for $\ell\geq1$ and for each $e'\in G$,
there exists a constant $C_\ell>0$ depending on $\ell$, but independent of $e'$, such that
\begin{equation}\label{equ 3.14}
\begin{aligned}
\overline{\mathbb{E}}\Bigg[\bigg(\int_0^t|\overline{\theta}_{s,t,e'}|^2ds
\bigg)^{\frac{\ell}{2}}+
\bigg(\int_0^t\int_G|\overline{\gamma}_{s,t,e'}(e)|^2
\lambda(de)ds\bigg)^{\frac{\ell}{2}}\Bigg]\leq C_\ell.
\end{aligned}
\end{equation}
Indeed, for $\ell\geq2$, from Burkholder-Davis-Gundy, Doob's maximal inequality and H\"{o}lder inequality, we have,
for $t\in[0,T],$
\begin{equation}\label{equ 3.15}
\begin{aligned}
&\overline{\mathbb{E}}\Big[\Big(\int_0^t|\overline{\theta}_{s,t,e'}|^2ds
+\int_0^t\int_G|\overline{\gamma}_{s,t,e'}(e)|^2\overline{N}(de,ds)\Big)^{\frac{\ell}{2}}\Big]\\
&\leq C_\ell \overline{\mathbb E}\Big[\sup_{s\in[0,t]}\Big| \int_0^s\overline{\theta}_{\tau,s,e'} d\overline{W}(\tau)+\int_0^s\int_G\overline{\gamma}_{\tau,s,e'}(e)\overline{N}_\lambda(de,d\tau)\Big|^\ell\Big]\\
&\leq C_\ell(\frac{\ell}{\ell-1})^\ell
\overline{\mathbb{E}}\Big[|\int_0^t\overline{\theta}_{s,t,e'}d\overline{W}(s)
+\int_0^t\int_G\overline{\gamma}_{s,t,e'} (e)\overline{N}_\lambda(de,ds)|^\ell \Big]\\
&\leq C_\ell \Big\{
\overline{\mathbb{E}}\Big[|\overline{\psi}_3(t,e')\overline{\psi}_2(t)\overline{n}(t)|^\ell\Big]
+|\mathbb{E}[\overline{\psi}_3(t,e')\overline{\psi}_2(t)\overline{n}(t)]|^\ell\Big\}\leq C_\ell,
\end{aligned}
\end{equation}
where $ C_\ell $  is independent of $e'$ because of (\ref{equ 3.12}).\\
Clearly, (\ref{equ 3.15}) implies, for each $e'\in G$,
$$ \overline{\mathbb{E}}\Big[\Big(\int_0^t\int_G|\overline{\gamma}_{s,t,e'}(e)|^2\overline{N}(de,ds)\Big)^{\frac{\ell}{2}}\Big]\leq C_\ell.$$
Recall Lemma 3.1 \cite {LW2}, it follows
\begin{equation}\label{equ 3.16}
\overline{\mathbb E}\Big[\Big(\int_0^t\int_G|\overline{\gamma}_{s,t,e'}(e)|^2\lambda(de)ds\Big)^{\frac{\ell}{2}}\Big]
\leq(\frac{\ell}{2})^{\frac{\ell}{2}} \overline{\mathbb {E}}\Big[\Big(\int_0^t\int_G|\overline{\gamma}_{s,t,e'}|^2\overline{N}(de,ds)\Big)^{\frac{\ell}{2}}\Big]\leq C_\ell.
\end{equation}
Hence, for $\ell\geq2,$ (\ref{equ 3.14}) holds true.\\
If $1\leq \ell<2$,  for each $e'\in G$, the fact $(\overline{\theta}_{\cdot,t,e'},\overline{\gamma}_{\cdot,t,e'})\in \mathcal{H}^2_\mathbb{\overline {F}}(0,t)\times \overline{K}_\lambda^2(0,t)$
and H\"{o}lder inequality allow to show (\ref{equ 3.14}).

We now estimate $\Xi_1^\varepsilon(t,e'), \Xi_2^\varepsilon(t,e'), \Xi_3^\varepsilon(t,e')$ one after another.

First, as for $\Xi_1^\varepsilon(t,e')$, following (\ref{equ 3.13}) we have
$$
\begin{aligned}
\Xi_1^\varepsilon(t,e')=\overline{\mathbb E}\Big[
\int_0^t\overline{\theta}_{s,t,e'}\Big(\overline{m}(s)\widetilde{\mathbb E}[\overline{\sigma}_\nu(s;\widetilde{X}^*(s))\widetilde{X}^{1,\varepsilon}(s)]
+\overline{m}(s)\delta\overline{\sigma}(s)\mathbbm{1}_{E_\varepsilon}(s)\Big)ds\Big].
\end{aligned}
$$
From this, the boundness of $\delta \sigma$  and H\"{o}lder inequality, it yields\\
$$
\begin{aligned}
&|\Xi_1^\varepsilon(t,e')|^2\\
&\leq 2\overline{\mathbb E}\Big[\Big(\int_0^t|\overline{\theta}_{s,t,e'}\overline{m}(s)
\widetilde{\mathbb E}[\overline{\sigma}_\nu(s;\widetilde{X}^*(s))\widetilde{X}^{1,\varepsilon}(s)]|ds\Big)^2\Big]
+2\overline{\mathbb E}\Big[\Big(\int_0^t|\overline{\theta}_{s,t,e'}\overline{m}(s)
\delta\overline{\sigma}(s)\mathbbm{1}_{E_\varepsilon}(s)|ds\Big)^2\Big]\\
&\leq 2\Big\{\overline{\mathbb E}[(\int_0^t|\overline{\theta}_{s,t,e'}|^2ds)^3]\Big\}^\frac{1}{3}\cdot
\Big\{\overline{\mathbb E}\Big[\sup_{s\in[0,T]}|\overline{m}(s)|^{12}\Big]\Big\}^\frac{1}{6}\cdot
\Big\{\overline{\mathbb E}[(\int_0^t|\widetilde{\mathbb E}
[\overline{\sigma}_\nu(s;\widetilde{X}^*(s))\widetilde{X}^{1,\varepsilon}(s)]|^3ds)^\frac{4}{3}]\Big\}^\frac{1}{2}\\
&\quad +2\varepsilon \Big\{\overline{\mathbb E}\Big[\sup_{s\in[0,T]}|\overline{m}(s)|^{4}\Big]\Big\}^\frac{1}{2}\cdot
\Big\{\overline{\mathbb E}[(\int_0^t|\mathbbm{1}_{E_\varepsilon}(s)\overline{\theta}_{s,t,e'}|^2ds)^2]\Big\}^\frac{1}{2}.
\end{aligned}
$$
Hence, thanks to (\ref{equ 3.7}), (\ref{equ 3.14}), there exists a constant $C>0$ independent of $e'$ such that
$$
\begin{aligned}
|\Xi_1^\varepsilon(t,e')|^4&\leq C \overline{\mathbb E}[\int_0^t|\widetilde{\mathbb E}
[\overline{\sigma}_\nu(s;\widetilde{X}^*(s))\widetilde{X}^{1,\varepsilon}(s)]|^4ds]+C\varepsilon^2\rho^*(\varepsilon),
\end{aligned}
$$
where $\rho^*(\varepsilon):=\overline{\mathbb E}[(\int_0^t|\mathbbm{1}_{E_\varepsilon}(s)\overline{\theta}_{s,t,e'}|^2ds)^2]$.
Obviously, the Dominated Convergence Theorem implies $\rho^*(\varepsilon)\rightarrow0$ as $\varepsilon\rightarrow0$.
Then it follows $\lambda(G)<\infty$ and (\ref{equ 3.4-11})-i)
that
\begin{equation}\label{equ 3.17}
\begin{aligned}
\int_0^r\int_G\mathbb{E}|\Xi_1^\varepsilon(t,e')|^4\lambda(de')dt
&\leq C\int_0^r
\Big(\int_0^t\overline{\mathbb E}[|\widetilde{\mathbb E}[\overline{\sigma}_\nu(s;\widetilde{X}^*(s))\widetilde{X}^{1,\varepsilon}(s)]|^4]ds\Big)dt
+C\varepsilon^2\rho^*(\varepsilon)
\leq \varepsilon^2\rho(\varepsilon).
\end{aligned}
\end{equation}

Second, we now pay attention to $\Xi_2^\varepsilon(t,e')$. Due to
$$
\Xi_2^\varepsilon(t,e')=\overline{\mathbb E}\bigg[\int_0^t\int_G\overline{\gamma}_{s,t,e'}(e)\overline{m}(s)
\widetilde{\mathbb E}[\overline{\beta}_\nu(s,e;\widetilde{X}^*(s))\widetilde{X}^{1,\varepsilon}(s)]\frac{1}{1+\overline{\beta}_x(s,e)}\lambda(de)ds
\bigg],
$$
and  from (\ref{equ 3.7}), (\ref{equ 3.14})  we get
$$
\begin{aligned}
|\Xi_2^\varepsilon(t,e')|^2&\leq C
\Big\{\overline{\mathbb E}[\int_0^t\int_G|\widetilde{\mathbb E}[\overline{\beta}_\nu(s,e;\widetilde{X}^*(s))
\widetilde{X}^{1,\varepsilon}(s)]|^4\lambda(de)ds]\Big\}^\frac{1}{2}\cdot\\
&\qquad\qquad\qquad\ \Big\{\overline{\mathbb E}[(\int_0^t\int_G|\overline{\gamma}_{s,t,e'}(e)|^2\lambda(de)ds)^4]\Big\}^\frac{1}{4}\cdot
\Big\{\overline{\mathbb E}\Big[\sup_{s\in[0,T]}|\overline{m}(s)|^{8}\Big]\Big\}^\frac{1}{4}.
\end{aligned}
$$
Hence,
\begin{equation}\label{equ 3.18}
\begin{aligned}
\int_0^r\int_G\mathbb E[|\Xi_2^\varepsilon(t,e')|^4]\lambda(de')dt
\leq C\int_0^r\overline{\mathbb E}[\int_0^t\int_G|\widetilde{\mathbb E}[\overline{\beta}_\nu(s,e;\widetilde{X}^*(s))
\widetilde{X}^{1,\varepsilon}(s)]|^4\lambda(de)ds]dt.
\end{aligned}
\end{equation}

Third, as for $\Xi_3^\varepsilon(t,e')$, since $|\psi_3(t,e')|\leq C(1\wedge|e'|)\leq C$, we have
\begin{equation}\label{equ 3.19}
\begin{aligned}
|\Xi_3^\varepsilon(t,e')|&\leq C\overline{\mathbb E}\bigg[|\overline{\psi}_2(t)\overline{\Theta}_3
^\varepsilon(t)|\bigg]\\
&\leq C \bigg\{ \overline{\mathbb E}[|\overline{\psi}_2(t)\overline{n}(t)\int_0^t(\overline{m}(s)\widetilde{\mathbb E}[\overline{b}_\nu(s;\widetilde{X}^*(s))\widetilde{X}^{1,\varepsilon}(s)]
+\overline{m}(s)\delta \overline{b}(s)\mathbbm{1}_{E_\varepsilon}(s))ds| ]\\
&+\overline{\mathbb E} [|\overline{\psi}_2(t)\overline{n}(t)
\int_0^t(\overline{m}(s)\overline{\sigma}_x(s)
\widetilde{\mathbb E} [\overline{\sigma}_\nu(s;\widetilde{X}^*(s)
)\widetilde{X}^{1,\varepsilon}(s) ]
+\overline{m}(s)\overline{\sigma}_x(s)\delta \overline{\sigma}(s)\mathbbm{1}_{E_\varepsilon}(s))ds| ]\\
&+\overline{\mathbb E} [|\overline{\psi}_2(t)
\overline{n}(t)\int_0^t\int_G\overline{m}(s)
\frac{\overline{\beta}_x(s,e)}{1+\overline{\beta}_x(s,e)}
\widetilde{\mathbb E}[\overline{\beta}_\nu(s,e;
\widetilde{X}^*(s))\widetilde{X}^{1,\varepsilon}
(s)]\lambda(de)ds| ]\bigg\}.
\end{aligned}
\end{equation}

The boundness of $b, \sigma, \sigma_x$, (\ref{equ 3.7}) and the assumption $\overline{\mathbb E}[\sup_{t\in[0,T]}|\overline{\psi}_2(t)|^{2\ell}]\leq C_\ell,\
\ell\geq1$
allow to show
\begin{equation}\label{equ 3.19-1}
\begin{aligned}
\overline{\mathbb E}\bigg[|\overline{\psi}_2(t)\overline{n}(t)
\int_0^t\overline{m}(s)(\delta \overline{b}(s)+ \overline{\sigma}_x(s)\delta \overline{\sigma}(s))
\mathbbm{1}_{E_\varepsilon}(s)ds|\bigg]\leq C\varepsilon.
\end{aligned}
\end{equation}
On the other hand, notice that
$$
\begin{aligned}
&\mathrm{i)}\ \overline{\mathbb E}\bigg[|\overline{\psi}_2(t)\overline{n}(t)
\int_0^t\overline{m}(s)
\widetilde{\mathbb E}[\overline{b}_\nu(s;\widetilde{X}^*(s))
\widetilde{X}^{1,\varepsilon}(s)]ds|\bigg]\\
&\quad\leq  \bigg\{
\overline{\mathbb E}\big[\int_0^t|\widetilde{\mathbb E}[\overline{b}_\nu(s;\widetilde{X}^*(s))
\widetilde{X}^{1,\varepsilon}(s)]|^2ds\big]
\bigg\}^{\frac{1}{2}}\cdot
\Big\{\overline{\mathbb E}\big[\sup_{t\in[0,T]}|\overline{\psi}_2(t)
\overline{n}(t)|^4\big]\Big\}^{\frac{1}{4}}\cdot
\Big\{\overline{\mathbb E}\big[\sup_{t\in[0,T]}|\overline{m}(t)|^4\big]\Big\}^{\frac{1}{4}}\\
&\quad\leq  C \Big\{
\overline{\mathbb E}\big[\int_0^t|\widetilde{\mathbb E}\big[\overline{b}_\nu(s;\widetilde{X}^*(s))
\widetilde{X}^{1,\varepsilon}(s)]|^2ds\big]\Big\}^{
\frac{1}{2}};\\
&\mathrm{ii)}\ \overline{\mathbb E}\Big[|\overline{\psi}_2(t)\overline{n}(t)
\int_0^t\overline{m}(s)\overline{\sigma}_x(s)
\widetilde{\mathbb E}[\overline{\sigma}_\nu
(s;\widetilde{X}^*(s))\widetilde{X}^{1,\varepsilon}(s)]ds|\Big]
\leq  C \Big\{
\overline{\mathbb E} [\int_0^t|\widetilde{\mathbb E} [\overline{\sigma}_
\nu(s;\widetilde{X}^*(s))\widetilde{X}^{1,
\varepsilon}(s) ]|^2ds\big]\Big\}^{\frac{1}{2}};\\
&\mathrm{iii)}\ \overline{\mathbb E}\Big[|\overline{\psi}_2(t)\overline{n}(t)
\int_0^t\int_G\overline{m}(s)
\frac{\overline{\beta}_x(s,e)}
{1+\overline{\beta}_x(s,e)}
\widetilde{\mathbb E}\big[\overline{\beta}_\nu(s,e;
\widetilde{X}^*(s))\widetilde{X}^{1,\varepsilon}(s)]\lambda(de)ds|\Big]\\
&\quad\leq C \Big\{
\overline{\mathbb E}[\int_0^t\int_G|\widetilde{\mathbb E}[\overline{\beta}_\nu(s,e;\widetilde{X}^*(s))
\widetilde{X}^{1,\varepsilon}(s)]|^2\lambda(de)ds\big]\Big\}^{\frac{1}{2}}.\\
\end{aligned}
$$
Combining (\ref{equ 3.19}), (\ref{equ 3.19-1}) and the above i), ii), iii), we know that there exists a constant $C>0$ independent of $e'$, such that
\begin{equation}\label{equ 3.20}
\begin{aligned}
|\Xi_3^\varepsilon(t,e')|^4&\leq C \varepsilon^4+C\overline{\mathbb E}[\int_0^t|\widetilde{\mathbb E}[\overline{b}_\nu(s;\widetilde{X}^*(s))\widetilde{X}^{1,\varepsilon}(s)]|^4ds]
+C\overline{\mathbb E}[\int_0^t|\widetilde{\mathbb E}[\overline{\sigma}_\nu(s;\widetilde{X}^*(s))\widetilde{X}^{1,\varepsilon}(s)]|^4ds]\\
&\quad+C\overline{\mathbb E}[\int_0^t\int_G|\widetilde{\mathbb E}[\overline{\beta}_\nu(s,e;\widetilde{X}^*(s))\widetilde{X}^{1,\varepsilon}(s)]|^4\lambda(de)ds].
\end{aligned}
\end{equation}
Hence, from (\ref{equ 3.4-11})-i)
\begin{equation}\label{equ 3.20-1}
\begin{aligned}
&\int_0^r\int_G\mathbb{E}|\Xi_3^\varepsilon(t,e')|^4\lambda(de')dt\\
&\quad\leq \varepsilon^2\rho(\varepsilon)+C\int_0^r\int_G
\mathbb{\overline{E}}\Big[\int_0^t\int_G|\widetilde{\mathbb E}[\overline{\beta}_\nu(s,e;\widetilde{X}^*(s))\widetilde{X}^{1,\varepsilon}(s)]|^4
\lambda(de)ds\Big]\lambda(de')dt\\
&\quad\leq\varepsilon^2\rho(\varepsilon)+C\lambda(G)\int_0^r\mathbb{\overline{E}}\Big[\int_0^t\int_G|\widetilde{\mathbb E}[\overline{\beta}_\nu(s,e;\widetilde{X}^*(s))\widetilde{X}^{1,\varepsilon}(s)]|^4
\lambda(de)ds\Big]dt.
\end{aligned}
\end{equation}

\noindent Thanks to (\ref{equ 3.11}), (\ref{equ 3.17}), (\ref{equ 3.18}), (\ref{equ 3.20-1}), we have
\begin{equation}\label{equ 3.21}
\begin{aligned}
&\mathbb E\Big[\int_0^r\int_G|\overline{\mathbb E}[\overline{\psi}_3(s,e')\overline{\psi}_2(s)\overline{X}^{1,\varepsilon}(s)]|^4\lambda(de')ds\Big]\\
&\leq\int_0^r\int_G \mathbb{E}|\Xi_1^\varepsilon(s,e')+\Xi_2^\varepsilon(s,e')+\Xi_3^\varepsilon(s,e') |^4\lambda(de')ds\\
&\leq \varepsilon^2\rho(\varepsilon)+C\int_0^r
\mathbb E\bigg[\int_0^t\int_G|\widetilde{\mathbb E}[\beta_\nu(s,e;\widetilde{X}^*(s))
\widetilde{X}^{1,\varepsilon}(s)]|^4\lambda(de)ds
\bigg]dt.
\end{aligned}
\end{equation}
Notice that $(\overline{\Omega},\overline{\mathcal{F}},\overline{P})$ is an intermediate probability space. So if we take
 $ \widetilde{\psi}_3(s,e)=\beta_\nu(s,e;\widetilde{X}^*(s)),\ \widetilde{\psi}_2(s)=1$, the Gronwall inequality can show,
 for $t\in[0,T]$,
 $$
 \mathbb E\bigg[\int_0^t\int_G|\widetilde{\mathbb E}[\beta_\nu(s,e;\widetilde{X}^*(s))
\widetilde{X}^{1,\varepsilon}(s)]
|^4\lambda(de)ds\bigg]\leq \varepsilon^2\rho(\varepsilon),
 $$
which and (\ref{equ 3.21}) imply the desired result, i.e., ii) of (\ref{equ 3.4-11}).
\end{proof}

\begin{corollary}\label{cor 3.3}
In (\ref{equ 3.4-11}), if taking $\widetilde{\psi}_2(t)=1$,
$ \widetilde{\psi}_1(t)=b_\nu(t;\widetilde{X}^*(t)), \sigma_\nu(t;\widetilde{X}^*(t))$ and
$\widetilde{\psi}_3(t,e)=\beta_\nu(t,e;\widetilde{X}^*(t))$, separately, one has
\begin{equation}\label{equ 3.23}
\begin{aligned}
 & \mathrm{i)}\quad \mathbb E\bigg[\int_0^T |\widetilde{\mathbb E}[b_\nu(s;\widetilde{X}^*(s))
\widetilde{X}^{1,\varepsilon}(s)]|^4ds\bigg]
+ \mathbb E\bigg[\int_0^T |\widetilde{\mathbb E}[\sigma_\nu(s;\widetilde{X}^*(s))
\widetilde{X}^{1,\varepsilon}(s)]|^4ds\bigg]
\leq \varepsilon^2\rho(\varepsilon);\\
 &\mathrm{ ii)}\quad  \mathbb E\bigg[\int_0^T\int_G|\widetilde{\mathbb E}[\beta_\nu(s,e;\widetilde{X}^*(s))
\widetilde{X}^{1,\varepsilon}(s)]|^4\lambda(de)ds
\bigg]\leq \varepsilon^2\rho(\varepsilon).
\end{aligned}
\end{equation}
\end{corollary}

\begin{proposition}\label{pro 3.4}
Let the Assumptions (A3.1)-(A3.4) hold true, then
\begin{equation}\label{equ 3.24}
\mathbb E[\sup_{t\in[0,T]}|X^\varepsilon(t)
-X^*(t)-X^{1,\varepsilon}(t)
-X^{2,\varepsilon}(t)|^{2}]
\leq\varepsilon^{2}\rho(\varepsilon).
\end{equation}
\end{proposition}

\begin{proposition}\label{pro 3.5}
Let us define
$$
\begin{aligned}
&(\phi_x,\phi_{xx})(T):=(\frac{\partial\phi}{\partial x},\frac{\partial^2\phi}{\partial x^2})(X^*(T),P_{X^*(T)}),
(\phi_\nu,\phi_{\nu a})(T;\widetilde{X}^*(T)):=(\frac{\partial\phi}{\partial \nu},\frac{\partial^2\phi}{\partial \nu \partial a})(X^*(T),P_{X^*(T)};\widetilde{X}^*(T)),
\end{aligned}
$$
and
\begin{equation}\label{equ 3.4-1}
\begin{aligned}
&\Lambda_T^\varepsilon(\phi):=\phi(X^\varepsilon(T),P_{X^\varepsilon(T)})-\phi(X^*(T),P_{X^*(T)})
 -\Big\{ \phi_x(T)(X^{1,\varepsilon}(T)+X^{2,\varepsilon}(T))\\
&+\widetilde{\mathbb E}[\phi_\nu(T;\widetilde{X}^*(T))(\widetilde{X}^{1,\varepsilon}(T)
+\widetilde{X}^{2,\varepsilon}(T))]
+\frac{1}{2}\phi_{xx}(T)(X^{1,\varepsilon}(T))^2+\frac{1}{2}\widetilde{\mathbb E}[\phi_{\nu a}(T;\widetilde{X}^*(T))(\widetilde{X}^{1,\varepsilon}(T))^2]\Big\},
\end{aligned}
\end{equation}
then\begin{equation}\label{equ 3.5}
\mathbb E[|\Lambda_T^\varepsilon(\phi)|^2]\leq \varepsilon^2\rho(\varepsilon).
\end{equation}
\end{proposition}

The similar proofs of the above  two propositions can be found in Buckdahn, Li and Ma \cite{BLM}.

\subsection{Adjoint equations}

In order to apply duality method to investigate our stochastic maximum principle, two adjoint equations are brought in. Compared
with the classical case, see Hu \cite{Hu}, a remarkable difference is that the first-order adjoint equation is a mean-field BSDE with
jumps. But the second-order adjoint equation is a classical linear BSDE with jumps, but not mean-field type.

\indent Let us first introduce some notations, which are used time and time again, for $\ell=x,y,z,k,$ $\theta=x,y,z,k,a_i$, $i,j=1,2$,
\begin{equation}\label{equ 3.27}
\begin{aligned}
&\Pi^*(s)=(X^*(s), Y^*(s), Z^*(s), K^*(s,\cdot)),\quad \Lambda^*(s)=(X^*(s),Y^*(s)),\\
&f_\ell(s)=\frac{\partial f}{\partial \ell}(s,\Pi^*(s), P_{ \Lambda^*(s)},u^*(s)),\\\
&(f_{\mu_i},f_{\mu_i \theta})(s;\widetilde{\Lambda}^*(s))=((\frac{\partial f}{\partial \mu})_i,
\frac{\partial}{\partial \theta}((\frac{\partial f}{\partial \mu})_i))(s,\Pi^*(s),
P_{ \Lambda^*(s)},u^*(s);\widetilde{\Lambda}^*(s)),\\
&(\widetilde{f}_{\mu_i},\widetilde{f}_{\mu_i \theta})(s)=((\frac{\partial f}{\partial \mu})_i,
\frac{\partial}{\partial \theta}((\frac{\partial f}{\partial \mu})_i))(s,\widetilde{\Pi}^*(s),
P_{ \Lambda^*(s)},\widetilde{u}^*(s);\Lambda^*(s)),\\
&f_{\mu_i \mu_j}(s;\widehat{\widetilde{\Lambda}}^*(s))=
(\frac{\partial}{\partial \mu}((\frac{\partial f}{\partial \mu})_i))_j
(s,\Pi^*(s), P_{ \Lambda^*(s)},u^*(s);\widetilde{\Lambda}^*(s), \widehat{\Lambda}^*(s)).\\
\end{aligned}
\end{equation}
With these concise notations in hand, the first-order adjoint equation can read as
\begin{equation}\label{equ 3.28}
  \left\{
   \begin{aligned}
-dY^1(s)&=F(s)ds-Z^1(s)dW(s)-\int_GR^1(s,e)N_\lambda(de,ds),\ s\in[0,T],\\
Y^1(T)&=\phi_x(T)+\widetilde{\mathbb E}[\widetilde{\phi}_\nu(T)],
\end{aligned}
   \right.
\end{equation}
where
$$
 \begin{aligned}
F(s)&=Y^1(s)\Big(f_y(s)+\widetilde{\mathbb E}[\widetilde{f}_{\mu_2}(s)]+f_z(s)\sigma_x(s)+\int_Gf_k(s)\beta_x(s,e)\lambda(de)+b_x(s)\Big)\\
&\quad+\widetilde{\mathbb E}\Big[\widetilde{Y}^1(s)\Big(\widetilde{f}_z(s)\widetilde{\sigma}_\nu(s)
+\int_G\widetilde{f}_k(s)\widetilde{\beta}_\nu(s,e)\lambda(de)+\widetilde{b}_\nu(s)\Big)\Big]\\
&\quad+Z^1(s)\Big(f_z(s)+\sigma_x(s)\Big)+\widetilde{\mathbb E}[\widetilde{Z}^1(s)\widetilde{\sigma}_\nu(s)]+f_x(s)+\widetilde{\mathbb E}[\widetilde{f}_{\mu_1}(s)]   \\
&\quad+\int_G\Big(R^1(s,e)\big(f_k(s)+\beta_x(s,e)\big)
+\widetilde{\mathbb E}[\widetilde{R}^1(s,e)\widetilde{\beta}_\nu(s,e)]\Big)\lambda(de),\\
\end{aligned}
$$
Under the Assumptions (A3.1)-(A3.2)  the unique solution of
 equation (\ref{equ 3.28}), $(Y^1,Z^1,R^1)$, satisfies, for $\ell\geq2$,
\begin{equation}\label{equ 3.29}
\mathbb E\bigg[\sup_{t\in[0,T]}|Y^1(t)|^\ell
+\bigg(\int_0^T|Z^1(t)|^2dt\bigg)^\frac{\ell}{2}
+\bigg(\int_0^T\int_G|R^1(t,e)|^2\lambda(de)dt\bigg)
^\frac{\ell}{2}\bigg]\leq C_\ell
\end{equation}
(see Proposition 4.1, Li \cite{Li}).
\quad\\
\indent
Once obtaining the solution $(Y^1,Z^1,R^1)$ of the equation (\ref{equ 3.28}), we
can consider the following  second-order adjoint equation
\begin{equation}\label{equ 3.30}
  \left\{
   \begin{aligned}
-dY^2(s)&=G(s)ds-Z^2(s)dW(s)-\int_GR^2(s,e)N_\lambda(de,ds),\ s\in[0,T],\\
Y^2(T)&=\phi_{xx}(T)+\widetilde{\mathbb E}[\widetilde{\phi}_{\nu y}(T)],
\end{aligned}
   \right.
\end{equation}
where
$$
\begin{aligned}
G(s)&=Y^2(s)\Big(f_y(s)+\widetilde{\mathbb E}[\widetilde{f}_{\mu_2}(s)]+2f_z(s)\sigma_x(s)+\int_Gf_k(s)\big(2\beta_x(s,e)+(\beta_x(s,e))^2\big)\lambda(de)\\
&\qquad\qquad\qquad\qquad\qquad\qquad\qquad\qquad\qquad  +2b_x(s)+(\sigma_x(s))^2+\int_G(\beta_x(s,e))^2\lambda(de)\Big)\\
\end{aligned}
$$
$$
\begin{aligned}
&\quad +Z^2(s)\Big(f_z(s)+2\sigma_x(s)\Big)+\int_GR^2(s,e)\Big(f_k(s)+2\beta_x(s,e)+(\beta_x(s,e))^2\Big)\lambda(de)\\
&\quad +Y^1(s)\Big(f_z(s)\sigma_{xx}(s)+\int_Gf_k(s)\beta_{xx}(s,e)\lambda(de)+b_{xx}(s)\Big)
+\widetilde{\mathbb E}\Big[\widetilde{Y}^1(s)\Big(\widetilde{f}_z(s)\widetilde{\sigma}_{\nu a}(s)\\
&\quad
+\int_G \widetilde{f}_k(s)\widetilde{\beta}_{\nu a}(s,e)\lambda(de)+\widetilde{b}_{\nu a}(s)\Big)\Big]
+Z^1(s)\sigma_{xx}(s)+\widetilde{\mathbb E}[\widetilde{Z}^1(s)\widetilde{\sigma}_{\nu a}(s)]\\
&\quad +\int_G\Big( R^1(s,e)\beta_{xx}(s,e)+\widetilde{\mathbb E}[\widetilde{R}^1(s,e)
\widetilde{\beta}_{\nu a}(s,e)]\Big)\lambda(de)+O(s)D^2f(s)O^\intercal(s)\\
&\quad+ \widetilde{\mathbb E}[\widetilde{f}_{\mu_1 a_1}(s)] +(Y^1(s))^2\widetilde{\mathbb E}[\widetilde{f}_{\mu_2 a_2}(s)],
\end{aligned}
$$
and  $O(s)=(1,Y^1(s), Y^1(s)\sigma_x(s)+Z^1(s),\int_G(Y^1(s)\beta_x(s,e)+R^1(s,e))\lambda(de))$,
$D^2f(s)$ denotes the Hessian matrix of $f$ with respect to $(x,y,z,k)$, i.e.,
$D^2f=\begin{pmatrix}
f_{xx}(t)&f_{xy}(t)&f_{xz}(t)& f_{xk}(t)\\
f_{yx}(t)&f_{yy}(t)&f_{yz}(t)&f_{yk}(t)\\
f_{zx}(t)&f_{zy}(t)&f_{zz}(t)&f_{zk}(t)\\
f_{kx}(t)&f_{ky}(t)&f_{kz}(t)&f_{kk}(t)\\
\end{pmatrix}.$
Since we have known $(Y^1,Z^1,R^1)$, the equation (\ref{equ 3.30}) is a classical linear BSDE with jumps. From the
well-known existence and uniqueness theorem of BSDEs with jumps,
under the Assumptions (A3.1)-(A3.3) the equation (\ref{equ 3.30})
admits a unique solution
$(Y^2,Z^2,R^2)$ and, moreover, for $\ell\geq2$,
\begin{equation}\label{equ 3.31}
\mathbb E\bigg[\sup_{t\in[0,T]}|Y^2(t)|^\ell
+
\bigg(\int_0^T|Z^2(t)|^2dt\bigg)^\frac{\ell}{2}
+\bigg(\int_0^T\int_G|R^2(t,e)|^2\lambda(de)dt
\bigg)^\frac{\ell}{2}\Bigg]\leq C_\ell.
\end{equation}

From Lemma \ref{le 3.2}, the following estimates hold true.
\begin{corollary}\label{cor 3.6}
Let the Assumptions (A3.1)-(A3.4) hold true, and set for $\ell=x,y,z,$
$$
\begin{aligned}
\widetilde{M}_1(s)&=(f_{\mu_1},f_{\mu_2},f_{\mu_1 \ell},f_{\mu_2 \ell})(s;\widetilde{\Lambda}^*(s)), \ \ \
\widetilde{M}_2(s)=(f_{k\mu_1},f_{k\mu_2})(s;\widetilde{\Lambda}^*(s)),\\
\widehat{\widetilde{M}}_3(s)&=(f_{\mu_1 \mu_1},f_{\mu_1 \mu_2},f_{\mu_2 \mu_2})(s;\widehat{\widetilde{\Lambda}}^*(s)).
\end{aligned}
$$
Moreover, let $X^{1,\varepsilon}$ and $Y^1$ be the solutions of (\ref{equ 3.2}) and  (\ref{equ 3.28}), respectively, then
\begin{equation}\label{equ 3.32}
\begin{aligned}
&\mathrm{i)}\ \mathbb E\Big[\int_0^T |\widetilde {\mathbb E}[\widetilde{M}_1(s)\widetilde{Y}^1(s)\widetilde{X}^{1,\varepsilon}(s)]|^4ds
+\int_0^T\int_G |\widetilde{{\mathbb E}}[\widetilde{M}_2(s)\widetilde{Y}^1(s)\widetilde{X}^{1,\varepsilon}(s)|^4 \lambda(de)ds\Big]\leq\varepsilon^2\rho(\varepsilon);\\
&\mathrm{ii)}\ \mathbb E\widehat{\mathbb E}\Big[\int_0^T |\widetilde{\mathbb E}[\widehat{\widetilde{M}}_3(s)\widetilde{Y}^1(s)\widetilde{X}^{1,\varepsilon}(s)]|^4ds\Big]\leq\varepsilon^2\rho(\varepsilon).\\
\end{aligned}
\end{equation}

\end{corollary}

\section{The second-order expansion of cost functional $Y^\varepsilon$}

 The second-order expansion of cost functional $Y^\varepsilon$ is stated in this section, which plays an important role in  proving our stochastic maximum principle. More precisely,
 we  prove that there exists a stochastic process $\breve{P}=(\breve{P}(t))_{t\in[0,T]}$ with $\breve{P}(T)=0$, such that, for all $t\in[0,T]$,
\begin{equation}\label{equ 4.1}
Y^\varepsilon(t)=Y^*(t)+Y^1(t)(X^{1,\varepsilon}(t)+X^{2,\varepsilon}(t))+\frac{1}{2}Y^2(t)(X^{1,\varepsilon}(t))^2
+\breve{P}(t)+o(\varepsilon),
\end{equation}
where the convergence is in $L^2(\Omega, C[0,T])$  sense.

For this purpose, let us first introduce the following linear mean-field BSDE with jumps:
\begin{equation}\label{equ 4.2}
  \left\{
\begin{aligned}
-d\breve{P}(t)&=\Big(f_y(t)\breve{P}(t)+f_z(t)\breve{Q}(t)+\int_Gf_k(t)\breve{K}(t,e)\lambda(de)
+\widetilde{\mathbb E}[f_{\mu_2}(t;\widetilde{\Lambda}^*(t))\widetilde{\breve{P}}(t)]\\
&\quad+\big(A_1(t)+\Delta f(t)\big)\mathbbm{1}_{E_\varepsilon}(t)\Big)dt
-\breve{Q}(t)dW(t)-\int_G\breve{K}(t,e)N_\lambda(de,dt), \  t\in[0,T],\\
\breve{P}(T)&=0,
\end{aligned}
\right.
\end{equation}
where
$$
\begin{aligned}
A_1(t)&=Y^1(t)\delta b(t)+Z^1(t)\delta \sigma(t)+\frac{1}{2}Y^2(t)(\delta\sigma(t))^2,\\
\Delta f(t)&=f(t,X^*(t),Y^*(t),Z^*(t)+Y^1(t)\delta\sigma(t),K^*(t,\cdot),P_{(X^*(t),Y^*(t))},v(t))\\
           &\quad-f(t,X^*(t),Y^*(t),Z^*(t),K^*(t,\cdot),P_{(X^*(t),Y^*(t))},u^*(t)).\\
\end{aligned}
$$
Obviously, under the Assumptions (A3.1)-(A3.2) the equation (\ref{equ 4.2}) possesses a unique solution
$(\breve{P},\breve{Q},\breve{K})\in \mathcal{S}^2_{\mathbb{F}}(0,T)\times \mathcal{H}^2_{\mathbb{F}}(0,T)\times K_\lambda^2(0,T)$
(see \cite{Li}). Moreover,
\begin{proposition}\label{pro 4.1}
Let the Assumptions (A3.1)-(A3.2) hold true, then for $\ell\geq2$,
\begin{equation}\label{equ 4.3}
E\Big[\sup_{t\in[0,T]}|\breve{P}(t)|^\ell+\Big(\int_0^T|\breve{Q}(t)|^2dt\Big)^{\frac{\ell}{2}}
+\Big(\int_0^T\int_G|\breve{K}(t,e)|^2\lambda(de)dt\Big)^{\frac{\ell}{2}}\Big]\leq \varepsilon^{\frac{\ell}{2}} \rho_\ell(\varepsilon),
\end{equation}
\end{proposition}
\noindent where $\rho_\ell:(0,+\infty)\rightarrow(0,+\infty)$ depending only on $\ell$ with $\rho_\ell(\varepsilon)\rightarrow$ as $\varepsilon\rightarrow0$.

\begin{proof}
From the standard argument for the solutions of classical BSDEs with jumps, we have, for $\ell\geq2$,
\begin{equation}\label{equ 4.4}
\begin{aligned}
&\mathbb E\Big[\sup_{t\in[0,T]}|\breve{P}(t)|^\ell+\Big(\int_0^T|\breve{Q}(t)|^2dt\Big)^{\frac{\ell}{2}}
+\Big(\int_0^T\int_G|\breve{K}(t,e)|^2\lambda(de)dt\Big)^{\frac{\ell}{2}}\Big]\\
&\leq C_\ell \mathbb E\bigg[(\int_0^T|A_1(t)+\Delta f(t)|\mathbbm{1}_{E_\varepsilon}(t)dt)^\ell
\bigg].
\end{aligned}
\end{equation}
The reader can refer to \cite{LW1}, \cite{LW2} for more details.\\
On the other hand, thanks to the boundness of $b,\ \sigma$ and the Lipschitz
property of $f$, H\"{o}lder inequality implies that
$$
\begin{aligned}
&\mathbb E\bigg[(\int_0^T|A(t)+\Delta f(t)|\mathbbm{1}_{E_\varepsilon}(t)dt)^\ell
\bigg]\\
&\leq C\mathbb E\Big[\Big(\int_0^T(|Y^1(s)\delta b(s)+Z^1(s)\delta\sigma(s)+\frac{1}{2}Y^2(s)(\delta\sigma(s))^2+Y^1(s)+1|\mathbbm{1}_{E_\varepsilon}(s))
ds\Big)^\ell\Big]\\
&\leq \varepsilon^\frac{\ell}{2}\rho_\ell(\varepsilon),
\end{aligned}
$$
where
$\rho_\ell(\varepsilon):=\varepsilon^{\frac{\ell}{2}}(E[\sup_{s\in[0,T]}|Y^1(s)|^\ell+\sup_{s\in[0,T]}|Y^2(s)|^\ell]+1)
+ \mathbb E\bigg[(\int_0^T|Z^1(s)|^2\mathbbm{1}
_{\mathbb E_\varepsilon}(s)ds)^\frac{\ell}{2}\bigg].$
Clearly, $\rho_\ell(\varepsilon)\rightarrow0$ as $\varepsilon\rightarrow0$. The proof is completed.
\end{proof}

The following theorem shows the second-order expansion of cost functional  $Y^\varepsilon$.
\begin{theorem}\label{th 4.2}
Let the Assumptions (A3.1)-(A3.4) hold true, then there exists a stochastic process over $[0,T]$,
$P=(\breve{P}(t))_{t\in[0,T]}$ with $\breve{P}(T)=0$, such that
\begin{equation}\label{equ 4.11}
\begin{aligned}
\mathbb E\Big[\sup_{t\in[0,T]}|Y^\varepsilon(t)-Y^*(t)-\breve{P}(t)-Y^1(t)(X^{1,\varepsilon}(t)+X^{2,\varepsilon}(t))
-\frac{1}{2}Y^2(t)(X^{1,\varepsilon}(t))^2|^2 \Big]\leq \varepsilon^2 \rho(\varepsilon).
\end{aligned}
\end{equation}
\end{theorem}

\begin{proof}
Like investigating classical Pontryagin Maximum Principle, an important element of proving Theorem \ref{th 4.2}
is to apply It\^{o}'s formula to
\begin{equation}\label{equ 4.7}
M(t):=Y^1(t)(X^{1,\varepsilon}(t)+X^{2,\varepsilon}(t))+\frac{1}{2}Y^2(t)(X^{1,\varepsilon}(t))^2.
\end{equation}
For this,  we have
\begin{equation}\label{equ 4.8}
\begin{aligned}
&Y^1(t)(X^{1,\varepsilon}(t)+X^{2,\varepsilon}(t))+\frac{1}{2}Y^2(t)(X^{1,\varepsilon}(t))^2
=Y^1(T)(X^{1,\varepsilon}(T)+X^{2,\varepsilon}(T))+\frac{1}{2}Y^2(T)(X^{1,\varepsilon}(T))^2\\
&-\int_t^T\Big(A(s)+A_4(s)\mathbbm{1}_{E_\varepsilon}(s)+A_5(s)\Big)ds
-\int_t^T\Big(B(s)+B_4(s)\mathbbm{1}_{E_\varepsilon}(s)+B_5(s)\Big)dW(s)\\
&-\int_t^T\int_G\Big(C^-(s,e)+C_4^-(s,e)\Big)N_\lambda(de,ds),\\
\end{aligned}
\end{equation}
where
$A,B,C, A_4,\cdot\cdot\cdot, C_4$ are given in Appendix.

Let us first admit the following lemma for a moment.
 Lemma   \ref{le 4.3}  argues  the powers of $\int_0^T|A_4(t)\mathbbm{1}_{E_\varepsilon}(t)|dt,$ $\int_0^T|A_5(t)|dt,$ $\int_0^T|B_4(t)
\mathbbm{1}_{E_\varepsilon}(t)|dt,$ $\int_0^T|B_5(t)|dt,$ $\int_0^T\int_G|C_4(t,e)|\lambda(de)dt$, as the
elements of $L^2(\Omega)$,
are  all $o(\varepsilon)$. Note that due to the structures of $A_4,A_5, B_4,B_5,C_4$ involving the first- and
second-order derivatives of the coefficients with respect to a measure, hence, the proof is not trivial and far from the classical case.
In the proof of Lemma \ref{le 4.3} we mainly borrow the new estimates given in Corollary {\ref{cor 3.3}} and Corollary {\ref{cor 3.6}}.
We place the proof of Lemma \ref{le 4.3} in Appendix.

\begin{lemma}\label{le 4.3}
We make the same Assumptions as in Theorem \ref{th 4.2}, then the following estimates hold true:
\begin{equation}\label{equ 4.9}
\begin{aligned}
&\mathrm{i)}\quad \mathbb  E\Big[\Big(\int_0^T|A_4(t)\mathbbm{1}_{E_\varepsilon}(t)|dt\Big)^2\Big]+
\mathbb E\Big[\Big(\int_0^T|A_5(t)|dt\Big)^2\Big]
\leq   \varepsilon^2\rho(\varepsilon);\\
&\mathrm{ii)}\quad \mathbb E\Big[\Big(\int_0^T|B_4(t)\mathbbm{1}_{E_\varepsilon}(t)|dt\Big)^2\Big]+
\mathbb E\Big[\Big(\int_0^T|B_5(t)|dt\Big)^2\Big]
\leq \varepsilon^2\rho(\varepsilon);\\
&\mathrm{iii)}\quad \mathbb E\Big[\Big(\int_0^T\int_G|C_4(t,e)|\lambda(de)dt\Big)^2\Big]
\leq  \varepsilon^2\rho(\varepsilon);\\
&\mathrm{iv)}\quad E\Big[
\int_0^T\Big(|M(t)|^2+|B(t)|^2+\int_G|C(t,e)|^2\lambda(de)\Big)\mathbbm{1}_{E_\varepsilon}(t)dt\Big]
\leq \varepsilon\rho(\varepsilon).\\
\end{aligned}
\end{equation}
\end{lemma}

With the help of Lemma \ref{le 4.3}, (\ref{equ 4.8}) can be written as
\begin{equation}\label{equ 4.10}
\begin{aligned}
&Y^1(t)(X^{1,\varepsilon}(t)+X^{2,\varepsilon}(t))+\frac{1}{2}Y^2(t)(X^{1,\varepsilon}(t))^2
=Y^1(T)(X^{1,\varepsilon}(T)+X^{2,\varepsilon}(T))+\frac{1}{2}Y^2(T)(X^{1,\varepsilon}(T))^2\\
&-\int_t^T A(s)ds-\int_t^TB(s)dW(s)-\int_t^T\int_GC^-(s,e)N_\lambda(de,ds)+o(\varepsilon).\\
\end{aligned}
\end{equation}

\noindent For convenience, let us set
\begin{equation}\label{equ 4.10-1}
\begin{aligned}
&\Delta X(t)=X^\varepsilon(t)-X^*(t)-X^{1,\varepsilon}(t)-X^{2,\varepsilon}(t),\quad
\Delta Y(t)=Y^\varepsilon(t)-Y^*(t)-\breve{P}(t)-M(t),\\
&\Delta Z(t)=Z^\varepsilon(t)-Z^*(t)-\breve{Q}(t)-B(t), \quad
\Delta K(t,e)=K^\varepsilon(t,e)-K^*(t,e)-\breve{K}(t,e)-C(t,e).
\end{aligned}
\end{equation}

\noindent By (\ref{equ 4.2}),   (\ref{equ 4.10}), (\ref{equ 4.10-1}) and the definition of $A(s)$, see Appendix, we have
we have
\begin{equation}\label{equ 4.12}
\begin{aligned}
\Delta Y(t)&=\int_t^T\Big\{f(s,X^\varepsilon(s),Y^\varepsilon(s),Z^\varepsilon(s),K^\varepsilon(s,\cdot),P_{(X^\varepsilon(s),Y^\varepsilon(s))},u^\varepsilon(s))\\
           &-f(s,X^*(s),Y^*(s),Z^*(s),K^*(s,\cdot),P_{(X^*(s),Y^*(s))},u^*(s))+A_2(s)+\frac{1}{2}A_3(s)\\
           &-\Big(f_y(s)\breve{P}(s)+f_z(s)\breve{Q}(s)+\int_Gf_k(s)\breve{K}(s,e)\lambda(de)
           +\widetilde{\mathbb E}[f_{\mu_2}(s;\widetilde{\Lambda}^*(s))\widetilde{\breve{P}}(s)]
           +\Delta f(s)\mathbbm{1}_{E_\varepsilon}(s)\Big)\Big\}ds\\
           &-\int_t^T\Delta Z(s)dW(s)-\int_t^T\int_G\Delta K(s,e)N_\lambda(ds,de)+o(\varepsilon),\quad t\in[0,T].
\end{aligned}
\end{equation}

We now analyse
$f(s,X^\varepsilon(s),Y^\varepsilon(s),Z^\varepsilon(s),K^\varepsilon(s,\cdot),P_{(X^\varepsilon(s),Y^\varepsilon(s))},u^\varepsilon(s))-
f(s,X^*(s),Y^*(s),Z^*(s),$  $K^*(s,\cdot),P_{(X^*(s),Y^*(s))},u^*(s))$.
To facilitate the presentation, let $\Lambda^\varepsilon(s)=(X^\varepsilon(s),Y^\varepsilon(s))$.
First, inspired by (\ref{equ 3.24}) and the definitions of $\Delta Y, \Delta Z, \Delta K$, we write\\
\begin{equation}\label{equ 3.29}
\begin{aligned}
&f(s,X^\varepsilon(s),Y^\varepsilon(s),Z^\varepsilon(s),K^\varepsilon(s,\cdot),P_{\Lambda^\varepsilon(s)},u^\varepsilon(s))-
f(s,X^*(s),Y^*(s),Z^*(s),K^*(s,\cdot),P_{\Lambda^*(s)},u^*(s))\\
&=\Delta f(s)\mathbbm{1}_{E_\varepsilon}(s)+I_1(s)+I_2(s),
\end{aligned}
\end{equation}
where
$$
\begin{aligned}
I_1(s)&=f(s,X^\varepsilon(s),Y^\varepsilon(s),Z^\varepsilon(s),K^\varepsilon(s,\cdot),P_{\Lambda^\varepsilon(s)},u^\varepsilon(s))
      -f(s,X^*(s)+X^{1,\varepsilon}(s)+X^{2,\varepsilon}(s),\\
      &\qquad Y^*(s)+\breve{P}(s)+M(s),Z^*(s)+\breve{Q}(s)+B(s), K^*(s,\cdot)+\breve{K}(s,\cdot)+C(s,\cdot), \\
      &\qquad P_{(X^*(s)+X^{1,\varepsilon}(s)+X^{2,\varepsilon}(s),Y^*(s)+\breve{P}(s)+M(s))},u^\varepsilon(s)),\\
 I_2(s)&= f(s,X^*(s)+X^{1,\varepsilon}(s)+X^{2,\varepsilon}(s),
     Y^*(s)+\breve{P}(s)+M(s),Z^*(s)+\breve{Q}(s)+B(s), \\
      &\qquad K^*(s,\cdot)+\breve{K}(s,\cdot)+C(s,\cdot),
      P_{(X^*(s)+X^{1,\varepsilon}(s)+X^{2,\varepsilon}(s),Y^*(s)+\breve{P}(s)+M(s))},u^\varepsilon(s))\\
      &-f(s,X^*(s),Y^*(s),Z^*(s)+Y^1(s)\delta\sigma(s)\mathbbm{1}_{E_\varepsilon}(s),K^*(s,\cdot),P_{\Lambda^*(s)},u^\varepsilon(s)).
      \end{aligned}
$$
Thanks to  the Lipschitz assumption on $f$ and the definitions of $\Delta Y, \Delta Z,\Delta K$ one  can obtain
\begin{equation}\label{equ 4.14}
\begin{aligned}
|I_1(s)|&\leq C\Big(|\Delta X(s)|+|\Delta Y(s)|+|\Delta Z(s)|+(\int_G|\Delta K(s,e)|^2\lambda(de))^\frac{1}{2}\\
&\quad+\big(\mathbb E[|\Delta X(s)|^2]\big)^{\frac{1}{2}}
+\big(\mathbb E[|\Delta Y(s)|^2]\big)^{\frac{1}{2}}\Big).
\end{aligned}
\end{equation}
Now focusing on $I_2(s)$. Obviously,
from the definition of $B(s)$, see Appendix,
$I_2(s)$ can be written as
$$
\begin{aligned}
I_2(s)&=I_3(s)+(I_4(s)-I_3(s))\mathbbm{1}_{E_\varepsilon}(s),
\end{aligned}
$$
here
$$
\begin{aligned}
I_3(s)&=f(s,X^*(s)+X^{1,\varepsilon}(s)+X^{2,\varepsilon}(s), Y^*(s)+\breve{P}(s)+M(s), Z^*(s)+\breve{Q}(s)+B_2(s)+\frac{1}{2}B_3(s),\\
      &\qquad K^*(s,\cdot)+\breve{K}(s,\cdot)+C(s,\cdot), P_{(X^*(s)+X^{1,\varepsilon}(s)+X^{2,\varepsilon}(s), Y^*(s)+\breve{P}(s)+M(s))},u^*(s))\\
      &\quad-f(s,X^*(s),Y^*(s),Z^*(s),K^*(s,\cdot),P_{\Lambda^*(s)},u^*(s)),\\
I_4(s)&=f(s,X^*(s)+X^{1,\varepsilon}(s)+X^{2,\varepsilon}(s), Y^*(s)+\breve{P}(s)+M(s), Z^*(s)+\breve{Q}(s)+B(s), K^*(s,\cdot)\\
           &\qquad\qquad\qquad\qquad\qquad\quad+\breve{K}(s,\cdot)+C(s,\cdot), P_{(X^*(s)+X^{1,\varepsilon}(s)+X^{2,\varepsilon}(s), Y^*(s)+\breve{P}(s)+M(s))},v(s))\\
           &\quad-f(s,X^*(s),Y^*(s),Z^*(s)+Y^1(s)\delta\sigma(s),K^*(s,\cdot),P_{(X^*(s),Y^*(s))},v(s)).
\end{aligned}
$$

\noindent According to Proposition \ref{pro 3.1}, Proposition \ref{pro 4.1}, Lemma \ref{le 4.3}-iv), the definition of $M, B_2, B_3, C$,
see Appendix, as well as the fact $W_2(P_\xi,P_\eta)\leq(E|\xi-\eta|^2)^{\frac{1}{2}},\ \xi,\eta\in L^2(\Omega,\mathcal{F}_T,P)$, we have
\begin{equation}\label{equ 4.15}
\begin{aligned}
&\mathbb E\bigg[\bigg(\int_0^T|I_4(s)-I_3(s)
|\mathbbm{1}_{E_\varepsilon}(s)ds\bigg)^2\bigg]\\
&\leq C\varepsilon \mathbb E\Big[\int_0^T
\bigg(|\breve{P}(s)+M(s)|^2+|\breve{Q}(s)
+B_2(s)+\frac{1}{2}B_3(s)|^2
+\int_G|\breve{K}(s,e)+C(s,e)|^2\lambda(de)\\
&\quad+|X^{1,\varepsilon}(s)
+X^{2,\varepsilon}(s)|^2
+\mathbb E[|X^{1,\varepsilon}(s)
+X^{2,\varepsilon}(s)|^2]
+\mathbb E[|\breve{P}(s)+M(s)|^2]\bigg)\mathbbm{1}_
{\mathbb E_\varepsilon}(s)ds\Big]\\
&\leq C\varepsilon^2\rho(\varepsilon)+C\varepsilon \mathbb E\bigg[\int_0^T(|M(s)|^2+|B_2(s)+
\frac{1}{2}B_3(s)|^2+\int_G|C(s,e)|^2
\lambda(de))\mathbbm{1}_{E_\varepsilon}(s)ds\bigg]\\
&\leq C\varepsilon^2\rho(\varepsilon).
\end{aligned}
\end{equation}

So we now slide to analyse $I_3(s)$. Applying the second-order expansion to $I_3(s)$, see Appendix for further details, we get\\
\begin{equation}\label{equ 4.16}
\begin{aligned}
I_3(s)&=f_y(s)\breve{P}(s)+f_z(s)\breve{Q}(s)+\int_Gf_k(s)\breve{K}(s,e)\lambda(de)+\widetilde{\mathbb E}[f_{\mu_2}(s;\widetilde{\Lambda}^*(s))\widetilde{\breve{P}}(s)] \\
&\quad+(X^{1,\varepsilon}(s)+X^{2,\varepsilon}(s))\Big(f_x(s)+f_y(s)Y^1(s)+f_z(s)(Y^1(s)\sigma_x(s)+Z^1(s))\\
&\quad+\int_Gf_k(s)(Y^1(s)\beta_x(s,e)+R^1(s,e))\lambda(de)\Big)
      +f_z(s)Y^1(s)\widetilde{\mathbb E}[\widetilde{\sigma}_\nu(s;\widetilde{X}^*(s))(\widetilde{X}^{1,\varepsilon}(s)+\widetilde{X}^{2,\varepsilon}(s))]\\
    &\quad+\int_Gf_k(s)Y^1(s)\widetilde{\mathbb E}[\widetilde{\beta}_\nu(s,e;\widetilde{X}^*(s))(\widetilde{X}^{1,\varepsilon}(s)
    +\widetilde{X}^{2,\varepsilon}(s))]\lambda(de)\\
    &\quad+\widetilde{\mathbb E}[f_{\mu_1}(s;\widetilde{\Lambda}^*(s))(\widetilde{X}^{1,\varepsilon}(s)+\widetilde{X}^{2,\varepsilon}(s))]
    +\widetilde{\mathbb E}[f_{\mu_2}(s;\widetilde{\Lambda}^*(s))\widetilde{Y^1}(s)(\widetilde{X}^{1,\varepsilon}(s)+\widetilde{X}^{2,\varepsilon}(s))]\\
    &\quad+\frac{1}{2}(X^{1,\varepsilon}(s))^2\bigg(f_y(s)Y^2(s)+f_z(s)\Big(Y^1(s)\sigma_{xx}(s)+2Y^2(s)\sigma_x(s)+Z^2(s)\Big)\\
   &\qquad\qquad\quad  +\int_G f_k(s)\big(Y^1(s)\beta_{xx}(s,e)+Y^2(s)(2\beta_x(s,e)+(\beta_x(s,e))^2)+R^2(s,e)\big)\lambda(de)\bigg)\\
    &\quad+\frac{1}{2}\Big(f_z(s)Y^1(s)\widetilde{\mathbb E}[\sigma_{\nu a}(s;\widetilde{X}^*(s))(\widetilde{X}^{1,\varepsilon}(s))^2]
    +\int_Gf_k(s)Y^1(s)\widetilde{\mathbb E}[\beta_{\nu a}(s,e;\widetilde{X}^*(s))(\widetilde{X}^{1,\varepsilon}(s))^2]\lambda(de)\\
   &\qquad\qquad\qquad\qquad\qquad\qquad\qquad\qquad\qquad\qquad\qquad\qquad\qquad
   +\widetilde{\mathbb E}[f_{\mu_2}(s;\widetilde{\Lambda}^*(s))\widetilde{Y}^1(s)(\widetilde{X}^{1,\varepsilon}(s))^2]\Big)\\
   &\quad+O(s)D^2f(s)O^\intercal(s)(X^{1,\varepsilon}(s))^2
   +\frac{1}{2}\widetilde{\mathbb E}\Big[\Big(f_{\mu_2a_2}(s;\widetilde{\Lambda}^*(s))(\widetilde{Y}^1(s))^2+f_{\mu_1a_1}(s;\widetilde{\Lambda}^*(s))\Big)
   (\widetilde{X}^{1,\varepsilon}(s))^2\Big]\\
   &\quad+I_{5}(s),
 \end{aligned}
\end{equation}
where $\mathbb E[(\int_0^TI_5(s)ds)^2]\leq \varepsilon^2 \rho(\varepsilon).$\\
Consequently, combing all the above analyses and recall the definitions of $A_2(s), A_3(s)$, see Appendix,
 the equation (\ref{equ 4.12}) can read as
\begin{equation}\label{equ 3.37}
\begin{aligned}
\Delta Y(t)&=\int_t^T\Big(I_1(s)+(I_4(s)-I_3(s))\mathbbm{1_{E_\varepsilon}}(s)+I_{5}(s)\Big)ds
-\int_t^T\Delta Z(s)dW(s)\\
&-\int_t^T\int_G\Delta K(s,e)N_\lambda(de,ds)+o(\varepsilon),\ t\in[0,T].
\end{aligned}
\end{equation}
It follows from (\ref{equ 4.14}), (\ref{equ 4.15}), (\ref{equ 4.16})
and Gronwall inequality that
$$
\begin{aligned}
&\mathbb E\Big[\sup_{s\in[0,T]}|\Delta Y(s)|^2+\int_0^T|\Delta Z(s)|^2ds+\int_0^T\int_G|\Delta K(s,e)|^2\lambda(de)ds\Big]\\
&\leq C\mathbb E\Big[(\int_0^T|(I_4(s)-I_3(s))\mathbbm{1_{E_\varepsilon}}(s)+I_5(s)|ds)^2\Big]+o(\varepsilon^2)\leq \varepsilon^2\rho(\varepsilon).
\end{aligned}
$$
The proof is completed.
\end{proof}

\begin{remark}\label{re 4.5}
If $f$ does not depend on $(y,z,k)$ and just depends on the law of $X^*(t)$, not on that of $Y^*(t)$, as well as
$\beta\equiv0$, then (\ref{equ 4.2}) is of the form
\begin{equation}\label{equ 6.1}
  \left\{
\begin{aligned}
-d\breve{P}(t)&=\Big(Y^1(t)\delta b(t)+Z^1(t)\delta \sigma(t)+\frac{1}{2}Y^2(t)(\delta\sigma(t))^2
+f(t,X^*(t),P_{X^*(t)},v(t))\\
&\quad-f(t,X^*(t),P_{X^*(t)},u^*(t))\Big)dt-\breve{Q}(t)dW(t),\ t\in[0,T],\\
\breve{Y}(T)&=0,
\end{aligned}
\right.
\end{equation}
which is just right  the case investigated by Buckdahn, Li and Ma \cite{BLM}, and, accordingly,
our stochastic maximum principle is consistent with theirs.

\end{remark}

\section{Stochastic maximum principle}

In this section, the second main result of this paper--SMP is proved.

\underline{\emph{Hamiltonian function}} We define
\begin{equation}\label{equ 3.38}
\begin{aligned}
&H(t,x,y,z,k,\nu,\mu,v;p,q,P)\\
&\quad=pb(t,x,\nu,v)+q\sigma(t,x,\nu,v)
+\frac{1}{2}P\Big(\sigma(t,x,\nu,v)-\sigma(t,X^*(t),P_{X^*(t)},u^*(t))\Big)^2\\
&\quad+f\Big(t,x,y,z+p\Big(\sigma(t,x,\nu,v)-\sigma(t,X^*(t),P_{X^*(t)},u^*(t))\Big), k,\mu, v\Big),
\end{aligned}
\end{equation}
$(t,x,y,z,k,\nu,\mu,v,p,q,P)\in[0,T]\times\mathbb{R}^3\times L^2(G,\mathscr{B}(G),\lambda)
\times \mathcal{P}_2(\mathbb{R})\times\mathcal{P}_2(\mathbb{R}^2)\times U \times\mathbb{R}^3.$\\

We now state the SMP.
\begin{theorem}
Let the Assumptions (A3.1)-(A3.4) hold true, and, furthermore, let
$$\widetilde{f}_{\mu_2}(t)=(\frac{\partial f}{\partial \mu})_2(t,\widetilde{X}^*(t),\widetilde{Y}^*(t),
\widetilde{Z}^*(t),\widetilde{K}^*(t,\cdot),P_{(X^*(t),Y^*(t))},\widetilde{u}^*(t);X^*(t),Y^*(t))>0,$$
$\ t\in[0,T],\ \widetilde{P}\otimes P$-a.s., and
$$f_k(t)=\frac{\partial f}{\partial k}(t,X^*(t), Y^*(t), Z^*(t), K^*(t,\cdot), P_{(X^*(t),Y^*(t))},u^*(t))>0,$$
$t\in[0,T], P$-a.s.
Suppose that $u^*(\cdot)$ is the
optimal control and $(X^*,Y^*,Z^*,K^*)$ is the corresponding solution of (\ref{equ 1.1}). Then there exist two pairs of stochastic
processes $(Y^1,Z^1,R^1)$ and $(Y^2,Z^2,R^2)$ satisfying (\ref{equ 3.28}) and (\ref{equ 3.30}), respectively, such that
\begin{equation}\label{equ 5.2}
\begin{aligned}
&H(t,X^*(t),Y^*(t),Z^*(t),K^*(t,\cdot),P_{X^*(t)}, P_{(X^*(t),Y^*(t))},v;Y^1(t),Z^1(t),Y^2(t))\\
&\geq H(t,X^*(t),Y^*(t),Z^*(t),K^*(t,\cdot),P_{X^*(t)}, P_{(X^*(t),Y^*(t))},u^*(t);Y^1(t),Z^1(t),Y^2(t)),
\end{aligned}
\end{equation}
\indent \qquad\qquad\qquad\qquad\qquad\qquad\qquad\qquad\qquad\qquad\qquad\qquad\qquad
 $\forall v\in U$,\  a.e.,\ a.s.
\end{theorem}

\begin{proof}
According to $J(v(\cdot))=Y^v(0)$, (\ref{equ 4.11}) and  $X^{1,\varepsilon}(0)=X^{2,\varepsilon}(0)=0$  we have
\begin{equation}\label{equ 5.3}
J(u^\varepsilon(\cdot))-J(u^*(\cdot))=Y^\varepsilon(0)-Y^*(0)=\breve{P}(0)+o(\varepsilon)\geq0.
\end{equation}
Recall that
\begin{equation}\label{equ 5.4}
  \left\{
   \begin{aligned}
-d\breve{P}(t)&=\Big\{f_y(t)\breve{P}(t)+f_z(t)\breve{Q}(t)+\int_Gf_k(t)\breve{K}(t,e)\lambda(de)
+\widetilde{\mathbb E}[f_{\mu_2}(t;\widetilde{\Lambda}^*(t))\widetilde{\breve{P}}(t)]\\
              &\quad+(A_1(t)+\Delta f(t))\mathbbm{1}_{E_\varepsilon}(t)\Big\}dt
              -\breve{Q}(t)dW(t)-\int_G\breve{K}(t,e)N_\lambda(de,dt),\ t\in[0,T],\\
\breve{P}(T)&=0,
\end{aligned}
   \right.
\end{equation}
which, however,  inspires us to consider the dual mean-field SDE with jumps:
\begin{equation}\label{equ 5.5}
  \left\{
   \begin{aligned}
d\Upsilon(t)&=\Big(f_y(t)\Upsilon(t)+\widetilde{\mathbb E}[\widetilde{f}_{\mu_2}(t)\widetilde{\Upsilon}(t)]\Big)dt
+f_z(t)\Upsilon(t)dW(t)+\int_Gf_k(t)\Upsilon(t)N_\lambda(de,dt),\ t\in[0,T],\\
\Upsilon(0)&=1.
\end{aligned}
   \right.
\end{equation}
Applying It\^{o}'s formula to $\Upsilon(t)\breve{P}(t),\ t\in[0,T]$, it follows
\begin{equation}\label{equ 5.6}
\breve{P}(0)=\mathbb E\Big[\int_0^T\Upsilon(s)(A_1(s)+\Delta f(s) )\mathbbm{1}_{E_\varepsilon}(s) ds\Big].
\end{equation}
But, $P$-a.s.
\begin{equation}\label{equ 5.6-1}
\Upsilon(s)>0,\ s\in[0,T].
\end{equation}
In fact,
consider the auxiliary BSDE with jumps
\begin{equation}\label{equ 5.7}
  \left\{
   \begin{aligned}
d\Upsilon^1(t)&=f_y(t)\Upsilon^1(t)dt+f_z(t)\Upsilon^1(t)dW(t)+\int_Gf_k(t)\Upsilon^1(t)N_\lambda(de,dt),\ t\in[0,T],\\
\Upsilon^1(0)&=1.
\end{aligned}
   \right.
\end{equation}
Denote  $\Delta\Upsilon(t)=\Upsilon^1(t)-\Upsilon(t),$ then from It\^{o}'s formula and the assumption $f_k(t)>0,\ t\in[0,T]$ we have
\begin{equation}\label{equ 5.8}
\begin{aligned}
 d((\Delta\Upsilon(t))^+)^2&=2\mathbbm{1}_{\{\Delta\Upsilon(t)>0\}} \Delta\Upsilon(t)\Big\{
 f_y(t)\Delta\Upsilon(t)-\widetilde{\mathbb E}[\widetilde{f}_{\mu_2}(t)\widetilde{\Upsilon}(t)]+f_z(t)\Delta \Upsilon(t)dW(t)\\
 &+\int_Gf_k(t)\Delta\Upsilon(t)N_\lambda(de,dt)\Big\}
 +\mathbbm{1}_{\{\Delta\Upsilon(t)>0\}}(f_z(t))^2(\Delta\Upsilon(t))^2dt\\
 &+\int_G\mathbbm{1}_{\{\Delta\Upsilon(t)>0\}}(f^-_k(t))^2(\Delta\Upsilon(t-))^2N(de,dt),
\end{aligned}
\end{equation}
where  $f^-_k(t)=\frac{\partial f}{\partial k}(s,\Pi^*(s-),P_{\Lambda^*(s)},u^*(s))$.\\
Moreover, notice $\widetilde{f}_{\mu_2}(s)>0,\ s\in[0,T],\ \widetilde{P}\otimes P$-a.s., and $\Upsilon^1(s)>0,\ s\in[0,T],\ P$-a.s., hence, for $s\in[0,T]$,
$\widetilde{\mathbb E}[\widetilde{f}_{\mu_2}(s)\widetilde{\Upsilon}^1(s)]>0$, and then
\begin{equation}\label{equ 5.9}
\begin{aligned}
\mathbb E[((\Delta\Upsilon(t))^+)^2]&\leq \mathbb E\Big[\int_0^t2\mathbbm{1}_{\{\Delta\Upsilon(s)>0\}}\Delta\Upsilon(s)
\{f_y(s)\Delta\Upsilon(s)+\widetilde{\mathbb E}[\widetilde{f}_{\mu_2}(s)\Delta\widetilde{\Upsilon}(s)]\}ds\Big]\\
 &\quad+\mathbb E\Big[\int_0^t\mathbbm{1}_{\{\Delta\Upsilon(s)>0\}}(f_z(s))^2(\Delta\Upsilon(s))^2ds\Big]\\
 &\quad+\mathbb E\Big[\int_0^t\int_G\mathbbm{1}_{\{\Delta\Upsilon(s)>0\}}(f_k(s))^2(\Delta\Upsilon(s))^2\lambda(de)ds\Big]\\
 &\leq C\mathbb E\Big[\int_0^t((\Delta\Upsilon(s))^+)^2ds\Big].
\end{aligned}
\end{equation}
Then the Gronwall lemma implies $\Delta\Upsilon(s)=\Upsilon^1(s)-\Upsilon(s)\leq0,\ s\in[0,T].$ On
the other hand, the assumption $f_k(t)>0, t\in[0,T]$ can show $\Upsilon^1(s)>0, s\in[0,T],\ P$-a.s. Hence,
$\Upsilon(s)\geq\Upsilon^1(s)>0,\ s\in[0,T],\ P$-a.s.
Combining (\ref{equ 5.3}), (\ref{equ 5.6}) and (\ref{equ 5.6-1}) we have the desired result.
\end{proof}

\section{Appendix}
\subsection{Some notations}

The aim of this subsection is to collect some notations used in this paper, in particular, in Section 4:
$$
\begin{aligned}
M(t)&=Y^1(t)(X^{1,\varepsilon}(t)+X^{2,\varepsilon}(t))+\frac{1}{2}Y^2(t)(X^{1,\varepsilon}(t))^2,\\
A_1(t)&=Y^1(t)\delta b(t)+Z^1(t)\delta\sigma(t)+\frac{1}{2}Y^2(t)(\delta\sigma(t))^2,\\
A_2(t) &=\Big(Y^1(t)b_x(t)+Z^1(t)\sigma_x(t)+\int_GR^1(t,e)\beta_x(t,e)\lambda(de)-F(t)\Big)(X^{1,\varepsilon}(t)+X^{2,\varepsilon}(t))\\
&\quad+\widetilde{\mathbb E}\Big[\Big(Y^1(t)b_{\nu}(t;\widetilde{X}^*(t))+Z^1(t)\sigma_{\nu}(t;\widetilde{X}^*(t))+\int_GR^1(t,e)
\beta_{\nu}(t,e;\widetilde{X}^*(t))\lambda(de)\Big)\\
&\qquad(\widetilde{X}^{1,\varepsilon}(t)+\widetilde{X}^{2,\varepsilon}(t))\Big],\\
A_3(t)
&=(X^{1,\varepsilon}(t))^2\Big(b_{xx}(t)Y^1(t)+\sigma_{xx}(t)Z^1(t)+\int_G\beta_{xx}(t,e)R^1(t,e)\lambda(de)+2Y^2(t)b_x(t)\\
&\quad+Y^2(t)(\sigma_x(t))^2+\int_GY^2(t)(\beta_x(t,e))^2\lambda(de)
+2Z^2(t)\sigma_x(t)+\int_GR^2(t,e)\big(2\beta_x(s,e)\\
&\quad+(\beta(t,e))^2\big)\lambda(de)-G(t)\Big)
+\widetilde{\mathbb E}\Big[\Big(Y^1(t)b_{\nu a}(t;\widetilde{X}^*(t))+Z^1(t) \sigma_{\nu a}(t;\widetilde{X}^*(t))\\
&\quad+\int_GR^1(t,e)\beta_{\nu a}(t,e;\widetilde{X}^*(t))\lambda(de)\Big) (\widetilde{X}^{1,\varepsilon}(t))^2 \Big],\\
A_4(t) &=X^{1,\varepsilon}(t)\Big\{Y^1(t)\delta b_x(t)+Z^1(t)\delta\sigma_x(t)+Y^2(t)\Big(\delta b(t)+\delta\sigma(t)\sigma_x(t)\Big)+Z^2(t)\delta\sigma(t)\Big\}\\
&\quad +\widetilde{\mathbb E}\Big[\widetilde{X}^{1,\varepsilon}(t)\Big\{Y^1(t)\delta b_\nu(t;\widetilde{X}^*(t))+Z^1(t)\delta\sigma_\nu(t;\widetilde{X}^*(t)) +Y^2(t)\delta\sigma(t)\sigma_\nu(t;\widetilde{X}^*(t))\Big\}\Big],\\
A_5(t)&=Y^2(t)X^{1,\varepsilon}(t)\widetilde{\mathbb E}[b_\nu(t;\widetilde{X}^*(t)) \widetilde{X}^{1,\varepsilon}(t)]
+\frac{1}{2}Y^2(t)\Big(\widetilde{\mathbb E}[\sigma_\nu(t;\widetilde{X}^*(t)) \widetilde{X}^{1,\varepsilon}(t)] \Big)^2\\
&\quad +Y^2(t)\sigma_x(t)X^{1,\varepsilon}(t)\widetilde{\mathbb E}[\sigma_\nu(t;\widetilde{X}^*(t))\widetilde{X}^{1,\varepsilon}(t)]
+Z^2(t)X^{1,\varepsilon}(t)\widetilde{\mathbb E}[\sigma_\nu(t;\widetilde{X}^*(t))\widetilde{X}^{1,\varepsilon}(t)]\\
&\quad  +\frac{1}{2}\int_G\Big(Y^2(t)\Big(\widetilde{\mathbb E}[\beta_\nu(t,e;\widetilde{X}^*(t))\widetilde{X}^{1,\varepsilon}(t)] \Big)^2
+Y^2(t)\beta_x(t,e)X^{1,\varepsilon}(t)\widetilde{\mathbb E}[\beta_\nu(t,e;\widetilde{X}^*(t))\\
&\qquad  \widetilde{X}^{1,\varepsilon}(t)]\Big)\lambda(de)+\int_G\Big(R^2(t,e)(\beta_x(t,e)+1)
X^{1,\varepsilon}(t)\widetilde{\mathbb E}[\beta_\nu(t,e;\widetilde{X}^*(t)) \widetilde{X}^{1,\varepsilon}(t)]\Big)\lambda(de),\\
B_1(t)&=Y^1(t)\delta \sigma(t),\\
B_2(t)&=(Y^1(t)\sigma_x(t)+Z^1(t))(X^{1,\varepsilon}(t)+X^{2,\varepsilon}(t))+Y^1(t)\widetilde{\mathbb E}
[\sigma_\nu(t;\widetilde{X}^*(t))(\widetilde{X}^{1,\varepsilon}(t)+\widetilde{X}^{2,\varepsilon}(t))],\\
B_3(t)&=(Y^1(t)\sigma_{xx}(t)+2Y^2(t)\sigma_x(t)+Z^2(t))(X^{1,\varepsilon}(t))^2+Y^1(t)\widetilde{\mathbb E}
[\sigma_{\nu a}(t,\widetilde{X}^*(t))(\widetilde{X}^{1,\varepsilon}(t))^2],\\
B_4(t)&=(Y^1(t)\delta\sigma_x(t)+Y^2(t)\delta\sigma(t))X^{1,\varepsilon}(t)
+\widetilde{\mathbb E}[Y^1(t)\delta\sigma_\nu(t;\widetilde{X}^*(t))\widetilde{X}^{1,\varepsilon}(t)],\\
B_5(t)&=Y^2(t)X^{1,\varepsilon}(t)\widetilde{\mathbb E}[\sigma_\nu(t,\widetilde{X}^*(t))\widetilde{X}^{1,\varepsilon}(t)],\\
C_2(t,e)&=\Big(Y^1(t)\beta_x(t,e)+R^1(t,e)\Big)(X^{1,\varepsilon}(t)+X^{2,\varepsilon}(t))
                   +Y^1(t)\widetilde{\mathbb E}[\beta_\nu(t,e;\widetilde{X}^*(t))(\widetilde{X}^{1,\varepsilon}(t)+\widetilde{X}^{2,\varepsilon}(t))],\\
\end{aligned}
$$
$$
\begin{aligned}
C_3(t,e)&=\Big(Y^1(t)\beta_{xx}(t,e)+Y^2(t)\big(2\beta_x(t,e)+(\beta_x(t,e))^2\big)+R^2(t,e)\Big)(X^{1,\varepsilon}(t))^2\\
         &\quad+Y^1(t)\widetilde{\mathbb E}[\beta_{\nu a}(t,e;\widetilde{X}^*(t))(\widetilde{X}^{1,\varepsilon}(t))^2],\\
  C_4(t,e)&=Y^2(t)X^{1,\varepsilon}(t)\widetilde{\mathbb E}[\beta_\nu(t,e;\widetilde{X}^*(t))\widetilde{X}^{1,\varepsilon}(t)]
+\frac{1}{2}Y^2(t)(\widetilde{\mathbb E}[\beta_\nu(t,e;\widetilde{X}^*(t))\widetilde{X}^{1,\varepsilon}(t)])^2\\
&\quad+Y^2(t)\beta_x(t,e)X^{1,\varepsilon}(t) \widetilde{\mathbb E}[\beta_\nu(t,e;\widetilde{X}^*(t))\widetilde{X}^{1,\varepsilon}(t)],
\end{aligned}
$$
and we denote
$$
\begin{aligned}
A(s)&=A_1(s)\mathbbm{1}_{E_\varepsilon}(s)+A_2(s)+\frac{1}{2}A_3(s),\ \
C(s,e)=C_2(s,e)+\frac{1}{2}C_3(s,e),\\
B(s)&=B_1(s)\mathbbm{1}_{E_\varepsilon}(s)+B_2(s)+\frac{1}{2}B_3(s).
\end{aligned}
$$
Here $F(t)$ and $G(t)$ are given in (\ref{equ 3.28}) and (\ref{equ 3.30}), respectively,
and  $C^-(s,e)$ denotes the time $s$ for the stochastic processes in $C(s,e)$ instead by $s-$.

\subsection{Proof of Lemma \ref{le 4.3}}

As for i) of (\ref{equ 4.9}), by observing the structures of $A_4(s)$ and $A_5(s)$, we mainly
work out the central ingredients, i.e.,  those terms involving the derivatives of the coefficients with respect to the measure.

a$_1$) From the boundness of $\sigma_\nu$, (\ref{equ 3.4}) and Dominated Convergence Theorem, it follows
\begin{equation}\label{equ 6.1}
\begin{aligned}
&\mathbb E\bigg[\bigg(\int_0^T\mathbbm{1}_{E_\varepsilon}(t)\widetilde{\mathbb E}[\widetilde{X}^{1,\varepsilon}
(t)Z^1(t)\delta\sigma_\nu
(t;\widetilde{X}^*(t))]dt\bigg)^2\bigg ]
\leq \varepsilon \mathbb E\bigg[\int_0^T\mathbbm{1}_{E_\varepsilon}(t)|Z^1(t)|^2\mathbb E[|X^{1,\varepsilon}(t)|^2]dt\bigg ]\\
&\leq \varepsilon \mathbb E\bigg[\sup_{t\in[0,T]}|X^{1,\varepsilon}(t)|^2\bigg ]\mathbb E\bigg[\int_0^T\mathbbm{1}_{E_\varepsilon}(t)|Z^1(t)|^2dt\bigg ]\leq\varepsilon^2\rho_1(\varepsilon),
\end{aligned}
\end{equation}
where $\rho_1(\varepsilon):=\mathbb E\bigg[\int_0^T\mathbbm{1}_{E_\varepsilon}(t)(Z^1(t))^2dt\bigg ]\rightarrow0,$ as $\varepsilon\rightarrow0.$

a$_2$) According to the boundness of $\sigma$, H\"{o}lder inequality and the estimates (\ref{equ 3.23}), (\ref{equ 3.31}), we obtain
\begin{equation}\label{equ 6.2}
\begin{aligned}
&\mathbb E\bigg[\bigg(\int_0^T\mathbbm{1}_
{E_\varepsilon}(t)\widetilde{\mathbb E}[\widetilde{X}^{1,\varepsilon}(t)Y^2(t)
\delta\sigma(t)\sigma_\nu(t;\widetilde{X}^*
(t)) ]dt\bigg)^2\bigg ]\\
&\leq C\varepsilon \mathbb E\bigg[\sup_{t\in[0,T]}|Y^2(t)|^2\int_0^T\mathbbm{1}_{E_\varepsilon}(t)|\widetilde{\mathbb E}[\widetilde{X}^{1,\varepsilon}(t)\sigma_\nu(t;\widetilde{X}^*(t))]|^2dt\bigg ]\\
&\leq C\varepsilon^{\frac{3}{2}} \Big\{\mathbb E[\sup_{t\in[0,T]}|Y^2(t)|^4 ]\Big\}^{\frac{1}{2}}
\Big\{\mathbb E\int_0^T|\widetilde{\mathbb E}[\widetilde{X}^{1,\varepsilon}(t)\sigma_\nu(t;\widetilde{X}^*(t))]|^4dt\Big\}^{\frac{1}{2}}\leq C\varepsilon^\frac{5}{2}\rho(\varepsilon).
\end{aligned}
\end{equation}

a$_3$) Thanks to the boundness of $\sigma_{\nu}$, (\ref{equ 3.4}) and (\ref{equ 3.23}), one can check
\begin{equation}\label{equ 6.3}
\begin{aligned}
&\mathbb E\bigg[\Big(\int_0^TY^2(t)
(\widetilde{\mathbb E}[\sigma_\nu(t;\widetilde{X}^*(t))
\widetilde{X}^{1,\varepsilon}(t)])^2dt
\Big)^2
\bigg]\\
&\leq CE\Big[(
\int_0^T|Y^2(t)|\mathbb E [|X^{1,\varepsilon}(t)| ]|
\widetilde{\mathbb E}[\sigma_\nu(t;\widetilde{X}^*(t))\widetilde{X}^{1,\varepsilon}(t)]|dt)^2\Big]\\
&\leq C \varepsilon \mathbb E\Big[\sup_{t\in[0,T]}|Y^2(t)|^2\cdot\int_0^T|\widetilde{\mathbb E}[\sigma_\nu(t;\widetilde{X}^*(t))\widetilde{X}^{1,\varepsilon}(t)]|^2dt\Big]\\
&\leq C \varepsilon \Big\{\mathbb E\Big[\sup_{t\in[0,T]}|Y^2(t)|^4\Big ]\Big\}^{\frac{1}{2}}
\Big\{\mathbb E [\int_0^T |\widetilde{\mathbb E}[\sigma_\nu(t;\widetilde{X}^*(t))
\widetilde{X}^{1,\varepsilon}(t) ] |^4dt ]\Big\}^{\frac{1}{2}}\\
&\leq C \varepsilon^2\rho(\varepsilon).
\end{aligned}
\end{equation}

a$_4$) The Assumptions (A3.1)-(A3.2), (\ref{equ 3.4})  and (\ref{equ 3.23})-ii) allow to show
\begin{equation}\label{equ 6.4}
\begin{aligned}
&\mathbb E\Big[ \Big(\int_0^T\int_GY^2(t)\beta_{x}(t,e)X^{1,\varepsilon}(t)\widetilde{\mathbb E}[\beta_{\nu}(t,e;\widetilde{X}^*(t)\widetilde{X}^{1,\varepsilon}(t))]\lambda(de)ds       \Big)^2\Big]\\
&\leq \mathbb E\Big[\sup_{t\in[0,T]}|X^{1,\varepsilon}(t)|^2\sup_{t\in[0,T]}|Y^2(t)|^2
\Big(\int_0^T\int_G(1\wedge|e|)|\widetilde{\mathbb E}[\beta_{\nu}(t,e;\widetilde{X}^*(t))\widetilde{X}^{1,\varepsilon}(t)]|\lambda(de)dt\Big)^2\Big]\\
&\leq  C \mathbb E\Big[\sup_{t\in[0,T]}|X^{1,\varepsilon}(t)|^2\sup_{t\in[0,T]}|Y^2(t)|^2
\Big(\int_0^T\int_G|\widetilde{\mathbb E}[\beta_{\nu}(t,e;\widetilde{X}^*(t)\widetilde{X}^{1,\varepsilon}(t))]|^4\lambda(de)dt\Big)^\frac{1}{2}\Big]\\
&\leq C  \Big\{\mathbb E[\sup_{t\in[0,T]}|X^{1,\varepsilon}(t)|^8]\Big\}^\frac{1}{4} \Big\{\mathbb E[\sup_{t\in[0,T]}|Y^2(t)|^8]\Big\}^\frac{1}{4}
\Big\{\mathbb E\int_0^T\int_G|\widetilde{\mathbb E}[\beta_{\nu}(t,e;\widetilde{X}^*(t))\widetilde{X}^{1,\varepsilon}(t)]|^4\lambda(de)ds\Big\}^\frac{1}{2}\\
&\leq\varepsilon^2\rho(\varepsilon).
\end{aligned}
\end{equation}

a$_5$) Notice that for each $e\in G,$ $|\beta_x(t,e)|\leq C(1\wedge|e|)\leq C$ and  $\lambda(G)<+\infty$,  (\ref{equ 3.4}) and (\ref{equ 3.23})
imply
\begin{equation}\label{equ 6.5}
\begin{aligned}
&\mathbb E\Big[\Big(\int_0^T\int_GR^2(t,e)(\beta_x(t,e)+1)X^{1,\varepsilon}(t)
\widetilde{\mathbb E}[\beta_\nu(t,e;\widetilde{X}^*(t))\widetilde{X}^{1,\varepsilon}(t)]\lambda(de)dt\Big)^2\Big]\\
&\leq C \mathbb E\Big[\sup_{t\in[0,T]}|X^{1,\varepsilon}(t)|^2
\Big(\int_0^T\int_G|R^2(t,e)\widetilde{\mathbb E}[\beta_\nu(t,e;\widetilde{X}^*(t))\widetilde{X}^{1,\varepsilon}(t)]|\lambda(de)dt\Big)^2
\Big]\\
&\leq C \Big\{\mathbb E\Big[\sup_{t\in[0,T]}|X^{1,\varepsilon}(t)|^8\Big]\Big\}^{\frac{1}{4}}
\Big\{\mathbb E\Big[\int_0^T\int_G|\widetilde{\mathbb E}[\beta_\nu(t,e;\widetilde{X}^*(t))\widetilde{X}^{1,\varepsilon}(t)]|^4\lambda(de)dt  \Big]\Big\}^\frac{1}{2}\\
&\leq C\varepsilon^2\rho(\varepsilon).
\end{aligned}
\end{equation}
Combining the above estimates a$_1$)-a$_5$), we have
\begin{equation}\label{equ 6.6}
\mathbb E\Big[\Big(\int_0^T|A_4(t)\mathbbm{1}_{E_\varepsilon}(t)|dt\Big)^2\Big]+
\mathbb E\Big[\Big(\int_0^T|A_5(t)|dt\Big)^2\Big]
\leq C \varepsilon^2\rho(\varepsilon).
\end{equation}

ii) of (\ref{equ 4.9}) can be calculated with the similar argument.

Let us now turn to  $C_4(s,e)$.  Through analysing the definition of $C_4(s,e)$,
in order to prove iii) in (\ref{equ 4.9})
we just need to estimate the following terms.

b$_1$)  By  (\ref{equ 3.4}), (\ref{equ 3.23}) and $\lambda(G)<+\infty$, one knows
\begin{equation}\label{equ 6.7}
\begin{aligned}
&\mathbb E\Big[\Big(
\int_0^T\int_GY^2(t)X^{1,\varepsilon}(t)\widetilde{\mathbb E}[\beta_\nu(t,e;\widetilde{X}^*(t))\widetilde{X}^{1,\varepsilon}(t)]\lambda(de)dt
\Big)^2\Big]\\
&\leq C \Big\{\mathbb E\Big[\sup_{t\in[0,T]}|Y^2(t)|^8\Big]\Big\}^{\frac{1}{4}}
\Big\{\mathbb E\Big[\sup_{t\in[0,T]}|X^{1,\varepsilon}(t)|^8\Big]\Big\}^{\frac{1}{4}}\cdot \\
&\qquad\qquad\qquad\qquad\qquad\qquad \qquad \Big\{\mathbb E\Big[\int_0^T\int_G|\widetilde{\mathbb E}[\beta_\nu(t,e;\widetilde{X}^*(t))\widetilde{X}^{1,\varepsilon}(t))]|^4\lambda(de)dt
\Big]\Big\}^{\frac{1}{2}}\\
&\leq C\varepsilon^2\rho(\varepsilon).
\end{aligned}
\end{equation}

b$_2$) On the other hand, from (\ref{equ 3.4}), (\ref{equ 3.23}) again, it also yields
\begin{equation}\label{equ 6.8}
\begin{aligned}
&\mathbb E\Big[\Big(\int_0^T\int_G
Y^2(t)\beta_x(t,e)X^{1,\varepsilon}(t)\widetilde{\mathbb E}[\beta_\nu(s,e;\widetilde{X}^*(t))\widetilde{X}^{1,\varepsilon}(t)]\lambda(de)dt
\Big)^2\Big]\\
\end{aligned}
\end{equation}
$$
\begin{aligned}
&\leq CE\Big[\sup_{t\in[0,T]}|Y^2(t)|^2\sup_{t\in[0,T]}|X^{1,\varepsilon}(t)|^2\Big(\int_0^T\int_G
(1\wedge|e|)|\widetilde{\mathbb E}[\beta_\nu(s,e;\widetilde{X}^*(t))\widetilde{X}^{1,\varepsilon}(t)]|\lambda(de)dt
\Big)^2  \Big]\\
&\leq C\varepsilon^2\rho(\varepsilon).
\end{aligned}
$$
According to the above estimates, iii) of (\ref{equ 4.9}) can be obtained.

We are now ready to investigate iv) of (\ref{equ 4.9}), i.e.,
$$
\mathbb E\Big[
\int_0^T(|M(s)|^2+|B(s)|^2+\int_G|C(s,e)|^2\lambda(de))\mathbbm{1}_{E_\varepsilon}(s)ds\Big]
\leq C \varepsilon\rho(\varepsilon).
$$

Recall the definitions of $M(s), B(s), C(s,e)$, it is feasible to consider some central estimates. Let us now show them one by one.

To begin with,
the following two estimates are the need for proving
$\mathbb E\Big[\int_0^T|B(s)|^2\mathbbm{1}_{E_\varepsilon}(s)ds\Big]\leq C \varepsilon\rho(\varepsilon)$.
From (\ref{equ 3.23}) and the boundness of $\sigma_{\nu a}$,     we have
\begin{equation}\label{equ 6.9}
\begin{aligned}
&\mathbb E\Big[\int_0^T\mathbbm{1}_{E_\varepsilon}(t)|Y^1(t)\widetilde{\mathbb E}
[\sigma_\nu(t;\widetilde{X}^*(t))\widetilde{X}^{1,\varepsilon}(t)]|^2dt\Big]\\
&\leq \mathbb E\Big[\sup_{t\in[0,T]}|Y^1(t)|^2\int_0^T\mathbbm{1}_{E_\varepsilon}(t)|\widetilde{\mathbb E}
[\sigma_\nu(t;\widetilde{X}^*(t))\widetilde{X}^{1,\varepsilon}(t)]|^2dt\Big]\\
&\leq \Big\{\mathbb E \big[\sup_{t\in[0,T]}|Y^1(t)|^4\big]\Big\}^{\frac{1}{2}}
\Big\{
\mathbb E\Big[(\int_0^T\mathbbm{1}_{E_\varepsilon}(t)|\widetilde{\mathbb E}
\big[\sigma_\nu(t;\widetilde{X}^*(t))\widetilde{X}^{1,\varepsilon}(t)\big]|^2dt)^2\Big]\Big\}^{\frac{1}{2}}\\
&\leq C \varepsilon^{\frac{1}{2}}
\Big\{\mathbb E\Big[\int_0^T|\widetilde{\mathbb E}
[\sigma_\nu(t;\widetilde{X}^*(t))\widetilde{X}^{1,\varepsilon}(t)]|^4dt\Big]\Big\}^{\frac{1}{2}}
\leq C \varepsilon^{\frac{3}{2}}\rho(\varepsilon),
\end{aligned}
\end{equation}
and
\begin{equation}\label{equ 6.10}
\begin{aligned}
&\mathbb E\Big[\int_0^T\mathbbm{1}_{E_\varepsilon}(t)|Y^1(t)|^2
|\widetilde{\mathbb E}[\sigma_{\nu a}(t;\widetilde{X}^*(t))(\widetilde{X}^{1,\varepsilon}(t))^2]|^2dt\Big]\\
&\leq\varepsilon \mathbb E\Big[ \sup_{t\in[0,T]}|Y^1(t)|^2\Big]\mathbb E\Big[ \sup_{t\in[0,T]}|X^{1,\varepsilon}(t)|^4\Big]\leq C\varepsilon^3.
\end{aligned}
\end{equation}

\indent What's more, let us show
$\mathbb E\Big[\int_0^T\int_G|C(t,e)|^2\mathbbm{1}_{E_\varepsilon}(t)\lambda(de)dt\Big]
\leq C \varepsilon\rho(\varepsilon).$
As in the preceding proof we are primarily concerned with expectation terms. As for
the expectation term in $C_2(t,e)$,  due to $\lambda(G)<+\infty$, (\ref{equ 3.23}) allows to show
\begin{equation}\label{equ 6.11}
\begin{aligned}
&\mathbb E\Big[ \int_0^T\int_G\mathbbm{1}_{E_\varepsilon}(t)|Y^1(t)\widetilde{\mathbb E}[\beta_\nu(t,e;\widetilde{X}^*(t))\widetilde{X}^{1,\varepsilon}]|^2\lambda(de)dt\Big]\\
&\leq \mathbb E\Big[\sup_{t\in[0,T]}|Y^1(t)|^2 \int_0^T\int_G\mathbbm{1}_{E_\varepsilon}(t)|\widetilde{\mathbb E}[\beta_\nu(t,e;\widetilde{X}^*(t))\widetilde{X}^{1,\varepsilon}]|^2\lambda(de)dt\Big]\\
&\leq\Big\{\mathbb E [\sup_{t\in[0,T]}|Y^1(t)|^4] \Big\}^{\frac{1}{2}}
\Big\{\mathbb E [ (\int_0^T\int_G\mathbbm{1}_{E_\varepsilon}(t)
|\widetilde{\mathbb E}[\beta_\nu(t,e;\widetilde{X}^*(t))\widetilde{X}^{1,\varepsilon}]|^2\lambda(de)dt)^2]\Big\}^{\frac{1}{2}}\\
&\leq   C\varepsilon^{\frac{1}{2}}  \Big\{\mathbb E[\int_0^T\int_G|\widetilde{\mathbb E}[\beta_\nu(t,e;\widetilde{X}^*(t))\widetilde{X}^{1,\varepsilon}]|^4\lambda(de)dt)]\Big\}^{\frac{1}{2}}\\
&\leq C\varepsilon^{\frac{3}{2}}\rho(\varepsilon).
\end{aligned}
\end{equation}
As regards the expectation term in $C_3(t,e)$, the boundness of $\beta_{\nu a}$ can imply

\begin{equation}\label{equ 6.12}
\begin{aligned}
&\mathbb E\Big[ \int_0^T\int_G\mathbbm{1}_{E_\varepsilon}(t)|Y^1(t)|^2
|\widetilde{\mathbb E}[\beta_{\nu a}(t,e;\widetilde{X}^*(t))(\widetilde{X}^{1,\varepsilon}(t))^2]|^2\lambda(de)dt\Big]\\
&\leq \varepsilon \Big\{\mathbb E[\sup_{t\in[0,T]}|Y^1(t)|^4]\Big\}^\frac{1}{2}\Big\{\mathbb E[\sup_{t\in[0,T]}|X^{1,\varepsilon}(t)|^8]\Big\}^\frac{1}{2}
\leq C\varepsilon^3.
\end{aligned}
\end{equation}

Finally,  the proof of  $\mathbb E\Big[\int_0^T|M(s)|^2\mathbbm{1}_{E_\varepsilon}(s)ds\Big]\leq  \varepsilon\rho(\varepsilon)$
is analogous to that of $\mathbb E\Big[\int_0^T|B(s)|^2\\
\mathbbm{1}_{E_\varepsilon}(s)ds\Big]\leq \varepsilon\rho(\varepsilon).$
So here we omit it. Combining all the above estimates, we can get iv) of (\ref{equ 4.9}). The proof is completed.

\subsection{The second order expansion of $I_3(s)$}

Making the first-order expansion of $f$ and according to  the definitions of $M(s), B_2(s), B_3(s), C(s,e)$, we obtain
\begin{equation}\label{equ 6.13}
\begin{aligned}
I_3(s)&=f_y(s)\breve{P}(s)+f_z(s)\breve{Q}(s)+\int_Gf_k(s)\breve{K}(s,e)\lambda(de)+\widetilde{\mathbb E}[f_{\mu_2}(s;\widetilde{\Lambda}^*(s))\widetilde{\breve{P}}(s)]\\
&+(X^{1,\varepsilon}(s)+X^{2,\varepsilon}(s))
\Big(f_x(s)+f_y(s)Y^1(s)+f_z(s)(Y^1(s)\sigma_x(s)+Z^1(s))\\
&+\int_Gf_k(s)(Y^1(s)\beta_x(s,e)+R^1(s,e))\lambda(de)\Big)
+f_z(s)Y^1(s)\widetilde{\mathbb E}[\sigma_{\nu}(s;\widetilde{X}^*(s))(\widetilde{X}^{1,\varepsilon}(s)+\widetilde{X}^{2,\varepsilon}(s))]\\
&+\widetilde{\mathbb E}[f_{\mu_1}(s;\widetilde{\Lambda}^*(s))(\widetilde{X}^{1,\varepsilon}(s)+\widetilde{X}^{2,\varepsilon}(s))]
+\widetilde{\mathbb E}[f_{\mu_2}(s;\widetilde{\Lambda}^*(s))\widetilde{Y}^1(s)(\widetilde{X}^{1,\varepsilon}(s)+\widetilde{X}^{2,\varepsilon}(s))]\\
&+\frac{1}{2}(X^{1,\varepsilon}(s))^2\Big(
f_y(s)Y^2(s)+f_z(s)(Y^1(s)\sigma_{xx}(s)+2Y^2(s)\sigma_x(s)+Z^2(s))\\
&+\int_Gf_k(s)(Y^1(s)\beta_{xx}(s,e)+Y^2(s)(2\beta_x(s,e)+(\beta_x(s,e))^2+R^2(s,e)))\lambda(de)\Big)\\
&+\frac{1}{2}f_z(s)Y^1(s)\widetilde{\mathbb E}[\sigma_{\nu a}(s;\widetilde{X}^*(s))(\widetilde{X}^{1,\varepsilon}(s))^2]
+\frac{1}{2}\int_Gf_k(s)Y^1(s)\widetilde{\mathbb E}[\beta_{\nu a}(s,e;\widetilde{X}^*(s))(\widetilde{X}^{1,\varepsilon}(s))^2]\lambda(de)\\
&+\frac{1}{2}\widetilde{\mathbb E}[f_{\mu_2}(s;\widetilde{\Lambda}^*(s))\widetilde{Y}^1(s)(\widetilde{X}^{1,\varepsilon}(s))^2]+\Delta(s),\\
\end{aligned}
\end{equation}
where
\begin{equation}\label{equ 6.14}
\begin{aligned}
\Delta(s)&=\int_0^1(f^\theta_x(s)-f_x(s))(X^{1,\varepsilon}(s)+X^{2,\varepsilon}(s))d\theta+\int_0^1(f^\theta_y(s)-f_y(s))(\breve{P}(s)+M(s))d\theta\\
&+\int_0^1(f^\theta_z(s)-f_z(s))(\breve{Q}(s)+B_2(s)+\frac{1}{2}B_3(s))d\theta
+\int_0^1\int_G(f^\theta_k(s)-f_k(s))(\breve{K}(s,e)\\
&+C(s,e))\lambda(de)d\theta
+\int_0^1\widetilde{\mathbb E}[(f^\theta_{\mu_1}(s;\widetilde{\Lambda}^*(s))-f_{\mu_1}(s;\widetilde{\Lambda}^*(s)))
(\widetilde{X}^{1,\varepsilon}(s)+\widetilde{X}^{2,\varepsilon}(s))]d\theta\\
&+\int_0^1\widetilde{\mathbb E}[(f^\theta_{\mu_2}(s;\widetilde{\Lambda}^*(s))-f_{\mu_2}(s;\widetilde{\Lambda}^*(s)))
(\widetilde{\breve{P}}(s)+\widetilde{M}(s))]d\theta,\\
\end{aligned}
\end{equation}
and  we define, for $0<\varrho<1,$ and $l=x,y,z,k$,
$$
\begin{aligned}
f^\varrho_l(s):&=\frac{\partial f}{\partial l}(s,X^*(s)+\varrho(X^{1,\varepsilon}(s)+X^{2,\varepsilon}(s)), Y^*(s)+\varrho(\breve{P}(s)+M(s)), \\
&\qquad Z^*(s)+\varrho(\breve{Q}(s)+B_2(s)+\frac{1}{2}B_3(s)),K^*(s,e)+\varrho(\breve{K}(s,e)+C(s,e)), \\
&\qquad P_{(X^*(s)+\varrho(X^{1,\varepsilon}(s)+X^{2,\varepsilon}(s)),
      Y^*(s)+\varrho(\breve{P}(s)+M(s)))},u^*(s)).
\end{aligned}
$$
$f^\theta_{\mu_1}(s;\widetilde{\Lambda}^*(s))$ and $f^\theta_{\mu_2}(s;\widetilde{\Lambda}^*(s))$ can be
understood in the same sense.

Let us now focus on $\Delta(s)$. To start with, we argue that
the first term on the right hand side can be written as
\begin{equation}\label{equ 6.15-1}
\begin{aligned}
&\bullet\quad \int_0^1(f_x^\theta(s)-f_x(s))(X^{1,\varepsilon}(s)+X^{2,\varepsilon}(s))d\theta
=\frac{1}{2}(X^{1,\varepsilon}(s))^2\Big(
f_{xx}(s)+f_{xy}(s)Y^1(s)\\
\end{aligned}
\end{equation}
$$
\begin{aligned}
&\qquad\qquad+f_{xz}(s)(Y^1(s)\sigma_x(s)+Z^1(s))+\int_Gf_{xk}(s)(Y^1(s)\beta_x(s,e)+R^1(s,e))\lambda(de)\Big)+I_1(s),
\end{aligned}
$$
where $\mathbb E\bigg[\int_0^T|I_1(s)|^2ds]\leq \varepsilon^2\rho(\varepsilon).$

In fact, according to Taylor expansion one has
\begin{equation}\label{equ 6.15}
\begin{aligned}
&\int_0^1(f_x^\theta(s)-f_x(s))(X^{1,\varepsilon}(s)+X^{2,\varepsilon}(s))d\theta
=\frac{1}{2}\Big\{
f_{xx}(s)(X^{1,\varepsilon}(s)+X^{2,\varepsilon}(s))^2+f_{xy}(s)(X^{1,\varepsilon}(s)
\\
&+X^{2,\varepsilon}(s))(\breve{P}(s)+M(s))
+f_{xz}(s)(X^{1,\varepsilon}(s)+X^{2,\varepsilon}(s))(\breve{Q}(s)+B_2(s)+\frac{1}{2}B_3(s))\\
&+\int_Gf_{xk}(s)(X^{1,\varepsilon}(s)+X^{2,\varepsilon}(s))(\breve{K}(s,e)+C(s,e))\lambda(de)\Big\}+\Theta_1(s)+\Theta_2(s),
\end{aligned}
\end{equation}
where
$$
\begin{aligned}
\Theta_1(s)&=\widetilde{\mathbb E}[f_{x\mu_1}(s;\widetilde{\Lambda}^*(s))(\widetilde{X}^{1,\varepsilon}(s)+\widetilde{X}^{2,\varepsilon}(s))]
(X^{1,\varepsilon}(s)+X^{2,\varepsilon}(s))\\
&\quad +\widetilde{\mathbb E}[f_{x\mu_2}(s;\widetilde{\Lambda}^*(s))(\widetilde{\breve{P}}(s)+\widetilde{M}(s))]
(X^{1,\varepsilon}(s)+X^{2,\varepsilon}(s)); \\
\Theta_2(s)&=\int_0^1\theta d\theta\int_0^1d\rho\Big\{
(f^{\rho\theta}_{xx}(s)-f_{xx}(s))(X^{1,\varepsilon}(s)+X^{2,\varepsilon}(s))^2\\
&\quad +(f_{xy}^{\rho\theta}(s)-f_{xy}(s))(X^{1,\varepsilon}(s)+X^{2,\varepsilon}(s))(\breve{P}(s)+M(s))\\
&\quad +(f_{xz}^{\rho\theta}(s)-f_{xz}(s))(X^{1,\varepsilon}(s)+X^{2,\varepsilon}(s))(\breve{Q}(s)+B_2(s)+\frac{1}{2}B_3(s))\\
&\quad +\int_G(f_{xk}^{\rho\theta}(s)-f_{xk}(s))(X^{1,\varepsilon}(s)+X^{2,\varepsilon}(s))(\breve{K}(s,e)+C(s,e))\lambda(de)
\Big\}.
\end{aligned}
$$

\noindent For proving (\ref{equ 6.15-1}), we will show six auxiliary estimates:
\begin{equation}\label{equ 6.16}
\begin{aligned}
&\mathrm{i)}\    \mathbb E[\int_0^T|f_{xx}(s)(X^{1,\varepsilon}(s)+X^{2,\varepsilon}(s))^2-f_{xx}(s)(X^{1,\varepsilon}(s))^2|^2ds]\leq\varepsilon^2\rho(\varepsilon);\\
&\mathrm{ii)}\   \mathbb E[\int_0^T|f_{xy}(s)(\breve{P}(s)+M(s))(X^{1,\varepsilon}(s)+X^{2,\varepsilon}(s))
-f_{xy}(s)Y^1(s)(X^{1,\varepsilon}(s))^2|^2ds]\leq\varepsilon^2\rho(\varepsilon);\\
&\mathrm{iii)}\   \mathbb E[\int_0^T|f_{xz}(s)(\breve{Q}(s)+B_2(s)+\frac{1}{2}B_3(s))(X^{1,\varepsilon}(s)+X^{2,\varepsilon}(s))\\
&\qquad\qquad\qquad\qquad\qquad\qquad\qquad\quad -f_{xz}(s)(Y^1(x)\sigma_x(s)+Z^1(s))(X^{1,\varepsilon}(s))^2|^2ds]\leq\varepsilon^2\rho(\varepsilon);\\
&\mathrm{iv)}\  \mathbb E[\int_0^T\int_G|f_{xk}(s)(\breve{K}(s,e)+C(s,e))(X^{1,\varepsilon}(s)+X^{2,\varepsilon}(s))\\
&\qquad\qquad\qquad\qquad\qquad\qquad\qquad -f_{xk}(s)(Y^1(x)\beta_x(s,e)+R^1(s,e))(X^{1,\varepsilon}(s))^2|^2\lambda(de)ds]\leq\varepsilon^2\rho(\varepsilon);\\
&\mathrm{v)}\  \mathbb E[\int_0^T|\Theta_1(s)|^2ds]\leq\varepsilon^2\rho(\varepsilon);\quad
\mathrm{vi)}\ \mathbb E[\int_0^T|\Theta_2(s)|^2ds]\leq\varepsilon^2\rho(\varepsilon).
\end{aligned}
\end{equation}

The proofs of i)-iii) are very analogous to that of iv). So it is now the main work to estimate iv)-vi) in (\ref{equ 6.16}).
As for iv), we are just concerned with the following expectation term because the other terms can be dealt with similarly.
Notice the boundness of $f_{xk}$  and (\ref{equ 3.23}), we get
\begin{equation}\label{equ 6.17}
\begin{aligned}
&\mathbb E\Big[\int_0^T\int_G|f_{xk}(s)X^{1,\varepsilon}(s)Y^1(s)
\widetilde{\mathbb E}[\beta_\nu(s,e;\widetilde{X}^*(s))\widetilde{X}^{1,\varepsilon}(s)]|^2\lambda(de)ds\Big]\\
&\leq C\mathbb E\Big[\sup_{s\in[0,T]}|X^{1,\varepsilon}(s)|^2\cdot\sup_{s\in[0,T]}|Y^1(s)|^2
\cdot\int_0^T\int_G
|\widetilde{\mathbb E}[\beta_\nu(s,e;\widetilde{X}^*(s))\widetilde{X}^{1,\varepsilon}(s)]|^2\lambda(de)ds\Big]\\
\end{aligned}
\end{equation}
$$
\begin{aligned}
&\leq C\Big\{ \mathbb E\Big[\sup_{s\in[0,T]}|X^{1,\varepsilon}(s)|^8\Big] \Big\}^\frac{1}{4}
\Big\{ \mathbb E\Big[\sup_{s\in[0,T]}|Y^1(s)|^8\Big] \Big\}^\frac{1}{4}\cdot\\
&\qquad\qquad \qquad \qquad \qquad \qquad \qquad \qquad
\Big\{ \mathbb E\Big[\int_0^T\int_G
|\widetilde{\mathbb E}[\beta_\nu(s,e;\widetilde{X}^*(s))\widetilde{X}^{1,\varepsilon}(s)]|^4\lambda(de)ds\Big] \Big\}^\frac{1}{2}\\
&\leq \varepsilon^2\rho(\varepsilon).
\end{aligned}
$$

Now observe $\Theta_1$,
the central work of proving v) is to estimate the following mean-field term,
\begin{equation}\label{equ 6.18}
\begin{aligned}
&\mathbb E\Big[\int_0^T|\widetilde{\mathbb E}[f_{x\mu_2}(s;\widetilde{Y}^*(s))\widetilde{Y}^1(s)\widetilde{X}^{1,\varepsilon}(s)]X^{1,\varepsilon}(s)|^2ds\Big]\\
&\leq C \mathbb E\Big[\sup_{s\in[0,T]}|X^{1,\varepsilon}(s)|^2\cdot
\int_0^T|\widetilde{\mathbb E}[f_{x\mu_2}(s;\widetilde{\Lambda}^*(s))\widetilde{Y}^1(s)\widetilde{X}^{1,\varepsilon}(s)]|^2ds\Big]\\
&\leq C\Big\{\mathbb E\Big[\sup_{s\in[0,T]}|X^{1,\varepsilon}(s)|^4 \Big]\Big\}^\frac{1}{2}
\Big\{\mathbb E\Big[\int_0^T|\widetilde{\mathbb E}[f_{x\mu_2}(s;\widetilde{\Lambda}^*(s))\widetilde{Y}^1(s)\widetilde{X}^{1,\varepsilon}(s)]|^4ds\Big]\Big\}^\frac{1}{2}\\
&\leq \varepsilon^2\rho(\varepsilon).
\end{aligned}
\end{equation}
The last step comes from (\ref{equ 3.4}) and (\ref{equ 3.32}).

For $\Theta_2$, since  the second-order derivatives of $f$ is Lipschitz continuous, hence, it is
easy to check that the power of $\varepsilon$  for each term of $\Theta_2$ is not less than $\frac{3}{2}$. So, vi) in (\ref{equ 6.16})
holds true.

Next, with the preceding  argument, we also have
\begin{equation}\label{equ 6.27}
\begin{aligned}
& \bullet\quad \int_0^1(f^\theta_y(s)-f_y(s))(\breve{P}(s)+M(s))d\theta
=\frac{1}{2}(X^{1,\varepsilon}(s))^2Y^1(s)\Big(
f_{xx}(s)+f_{xy}(s)Y^1(s)\\
&\qquad\qquad+f_{xz}(s)(Y^1(s)\sigma_x(s)+Z^1(s))+\int_Gf_{xk}(s)(Y^1(s)\beta_x(s,e)+R^1(s,e))\lambda(de)\Big)+I_2(s),\\
&\bullet\quad \int_0^1(f^\theta_z(s)-f_z(s))(\breve{Q}(s)+B_2(s)+\frac{1}{2}B_3(s))d\theta\\
&\qquad=\frac{1}{2}(X^{1,\varepsilon}(s))^2(Y^1(s)\sigma_x(s)+Q^1(s))
\Big(f_{xx}(s)+f_{xy}(s)Y^1(s)
+f_{xz}(s)(Y^1(s)\sigma_x(s)+Z^1(s))\\
&\qquad\qquad\qquad\qquad\qquad\qquad\qquad\qquad\quad+\int_Gf_{xk}(s)(Y^1(s)\beta_x(s,e)+R^1(s,e))\lambda(de)\Big)+I_3(s),\\
&\bullet\quad \int_0^1\int_G(f^\theta_k(s)-f_k(s))(\breve{K}(s,e)+C(s,e))\lambda(de)d\theta\\
&\qquad=\frac{1}{2}(X^{1,\varepsilon}(s))^2\int_G(Y^1\beta_x(s,e)+R^1(s,e))\lambda(de)
\Big(f_{xx}(s)+f_{xy}(s)Y^1(s)+f_{xz}(s)(Y^1(s)\sigma_x(s)\\
&\qquad\qquad\qquad\qquad\qquad\qquad\qquad+Z^1(s))+\int_Gf_{xk}(s)(Y^1(s)\beta_x(s,e)+R^1(s,e))\lambda(de)\Big)+I_4(s),
\end{aligned}
\end{equation}
where $\mathbb E [\int_0^T|I_2(s)|^2+|I_3(s)|^2+|I_4(s)|^2ds]\leq \varepsilon^2\rho(\varepsilon).$

In addition, we now switch to  analysing the mean-field term
$\int_0^1\widetilde{\mathbb E}[(f^\theta_{\mu_1}(s;\widetilde{\Lambda}^*(s))-f_{\mu_1}(s;\widetilde{\Lambda}^*(s)))\\
(\widetilde{X}^{1,\varepsilon}(s)+\widetilde{X}^{2,\varepsilon}(s))]d\theta.$
It follows Taylor expansion that
\begin{equation}\label{equ 6.28}
\begin{aligned}
&\bullet\ \int_0^1\widetilde{\mathbb E}[(f^\theta_{\mu_1}(s;\widetilde{\Lambda}^*(s))-f_{\mu_1}(s;\widetilde{\Lambda}^*(s)))
(\widetilde{X}^{1,\varepsilon}(s)+\widetilde{X}^{2,\varepsilon}(s))]d\theta
=\frac{1}{2}\widetilde{\mathbb E}[f_{\mu_1 a_1}(s;\widetilde{\Lambda}^*(s))(\widetilde{X}^{1,\varepsilon}(s))^2]+I_5(s),
\end{aligned}
\end{equation}
where
$$
\begin{aligned}
I_5(s)&=I_{5,1}(s)+I_{5,2}(s)+I_{5,3}(s),\\
I_{5,1}(s)&=\frac{1}{2}\widetilde{\mathbb E}\Big[f_{\mu_1 a_1}(s;\widetilde{\Lambda}^*(s))
\Big((\widetilde{X}^{1,\varepsilon}(s)+\widetilde{X}^{2,\varepsilon}(s))^2-(\widetilde{X}^{1,\varepsilon}(s))^2\Big)\Big],\\
I_{5,2}(s)&=\frac{1}{2}\Big\{\widetilde{\mathbb E}\Big[
(\widetilde{X}^{1,\varepsilon}(s)+\widetilde{X}^{2,\varepsilon}(s))
\Big(f_{\mu_1x}(s;\widetilde{\Lambda}^*(s))(X^{1,\varepsilon}(s)+X^{2,\varepsilon}(s))+ f_{\mu_1y}(s;\widetilde{\Lambda}^*(s))(\breve{P}(s)+M(s))\\
&\quad+f_{\mu_1z}(s;\widetilde{\Lambda}^*(s))(\breve{Q}(s)+B_2(s)+\frac{1}{2}B_3(s))
+\int_Gf_{\mu_1k}(s;\widetilde{\Lambda}^*(s))(\breve{K}(s,e)+C(s,e))\lambda(de)\\
&\quad+\widehat{\mathbb{E}}[f_{\mu_1\mu_1}(s;\widehat{\widetilde{\Lambda}}^*(s))(\widehat{X}^{1,\varepsilon}(s)+\widehat{X}^{2,\varepsilon}(s))]
+\widehat{\mathbb{E}}[f_{\mu_1\mu_2}(s;\widehat{\widetilde{\Lambda}}^*(s))(\widehat{\breve{P}}(s)+\widehat{M}(s))]
\Big)\Big]\Big\},\\
I_{5,3}(s)&=\int_0^1\theta d\theta\int_0^1d\rho\Big\{
\widetilde{\mathbb E}\Big[(\widetilde{X}^{1,\varepsilon}(s)+\widetilde{X}^{2,\varepsilon}(s))
\Big( (f^{\rho\theta}_{\mu_1 x}(s;\widetilde{\Lambda}(s))-f_{\mu_1 x}(s;\widetilde{\Lambda}(s)))(X^{1,\varepsilon}(s)+X^{2,\varepsilon}(s))\\
&\quad + (f^{\rho\theta}_{\mu_1 y}(s;\widetilde{\Lambda}(s))-f_{\mu_1 y}(s;\widetilde{\Lambda}(s)))(\breve{P}(s)+M(s))\\
&\quad +(f^{\rho\theta}_{\mu_1 z}(s;\widetilde{\Lambda}(s))-f_{\mu_1 z}(s;\widetilde{\Lambda}(s)))(\breve{Q}(s)+B_2(s)+\frac{1}{2}B_3(s))\\
&\quad +\int_G(f^{\rho\theta}_{\mu_1 k}(s;\widetilde{\Lambda}(s))-f_{\mu_1 k}(s;\widetilde{\Lambda}(s)))(\breve{K}(s,e)+C(s,e))\lambda(de)\\
& \quad+\widehat{\mathbb{E}} [(f^{\rho\theta}_{\mu_1\mu_1}(s;\widehat{\widetilde{\Lambda}}^*(s))
-f_{\mu_1\mu_1}(s;\widehat{\widetilde{\Lambda}}^*(s))(\widehat{X}^{1,\varepsilon}(s)+\widehat{X}^{2,\varepsilon}(s)) ]\\
& \quad+\widehat{\mathbb E} [(f^{\rho\theta}_{\mu_1\mu_2}(s;\widehat{\widetilde{\Lambda}}^*(s))
-f_{\mu_1\mu_2}(s;\widehat{\widetilde{\Lambda}}^*(s))(\widehat{\breve{P}}(s)+\widehat{M}(s)) ]\\
& \quad+(f^{\rho\theta}_{\mu_1 a_1}(s;\widetilde{\Lambda}(s))-f_{\mu_1 a_1}(s;\widetilde{\Lambda}(s))) (\widetilde{X}^{1,\varepsilon}(s)
+\widetilde{X}^{2,\varepsilon}(s))
\Big)\Big]\Big\}.
\end{aligned}
$$
We now want to show $\mathbb E\bigg[\bigg(\int_0^TI_5(s)ds\bigg)^2\bigg]\leq \varepsilon^2\rho(\varepsilon)$. Indeed,
 From the boundness of $f_{\mu_1 a_1}$ and (\ref{equ 3.4}), it is easy to get
$\mathbb E\bigg[\bigg(\int_0^TI_{5,1}(s)ds\bigg)^2\bigg]\leq  C\varepsilon^3.$
Next let us prove $\mathbb E\bigg[\bigg(\int_0^TI_{5,2}(s)ds\bigg)^2\bigg]\leq\varepsilon^2\rho(\varepsilon).$
According to the structure of $I_{5,2}(s)$, we know that it is enough to only handle the following jump term and expectation term:
 $$
 \begin{aligned}
&\mathrm{i)}\ \mathbb E\Big[\Big(\int_0^T\int_G \widetilde{\mathbb E}[
f_{\mu_1k}(s;\widetilde{\Lambda}^*(s))\widetilde{X}^{1,\varepsilon}(s)](\breve{K}(s,e)+C(s,e))\lambda(de)ds\Big)^2\Big]\leq \varepsilon^2\rho(\varepsilon),\\
&\mathrm{ ii)}\ \mathbb E\Big[\Big(\int_0^T\widetilde{\mathbb E}\Big[ \widetilde{X}^{1,\varepsilon}(s)
\widehat{\mathbb{E}}[f_{\mu_1\mu_2}(s;\widetilde{X}^*(s),\widehat{Y}^*(s))(\widehat{\breve{P}}(s)+\widehat{M}(s))\Big]
ds\Big)^2  \Big]\leq   \varepsilon^2\rho(\varepsilon).
\end{aligned}
$$

For i), H\"{o}lder inequality, (\ref{equ 4.3}), (\ref{equ 3.4}), (\ref{equ 3.4-11})-ii) with
$\widetilde{\psi}_3(t)=f_{\mu_1 k}(s;\widetilde{\Lambda}^*(s)),\ \widetilde{\psi}_2(t)\equiv1$,
the definition of $C(s,e)$
and the assumption $\lambda(G)<+\infty$, we have\\
$$
\begin{aligned}
&\mathbb E\Big[\Big(\int_0^T\int_G \widetilde{\mathbb E}[f_{\mu_1k}(s;\widetilde{\Lambda}^*(s))\widetilde{X}^{1,\varepsilon}(s)](\breve{K}(s,e)+C(s,e))\lambda(de)ds\Big)^2\Big]\\
&\leq \mathbb E\Big[ \int_0^T\int_G |\widetilde{\mathbb E}[f_{\mu_1k}(s;\widetilde{\Lambda}^*(s))\widetilde{X}^{1,\varepsilon}(s)]|^2\lambda(de)ds\cdot
 \int_0^T\int_G |\breve{K}(s,e)+C(s,e)|^2\lambda(de)ds\Big] \\
 &\leq \Big\{\mathbb E\Big[(\int_0^T\int_G |\widetilde{\mathbb E}[f_{\mu_1k}(s;\widetilde{\Lambda}^*(s))\widetilde{X}^{1,\varepsilon}(s)]|^2\lambda(de)ds)^2\Big]\Big\}^{\frac{1}{2}}
 \Big\{ \mathbb E\Big[ (\int_0^T\int_G |\breve{K}(s,e)+C(s,e)|^2\lambda(de)ds)^2\Big]\Big\}^{\frac{1}{2}} \\
 &\leq \Big\{\mathbb E\Big[\int_0^T\int_G |\widetilde{\mathbb E}[f_{\mu_1k}(s;\widetilde{\Lambda}^*(s))\widetilde{X}^{1,\varepsilon}(s)]|^4\lambda(de)ds\Big]\Big\}^{\frac{1}{2}}\cdot
   \Big\{ \mathbb E\Big[ (\int_0^T\int_G |\breve{K}(s,e)+C(s,e)|^2\lambda(de)ds)^2\Big]\Big\}^{\frac{1}{2}}\\
\end{aligned}
$$
 $$
\begin{aligned}
&\leq \varepsilon\rho(\varepsilon) \Big(\varepsilon\rho(\varepsilon)
 + \Big\{ \mathbb E\Big[ (\int_0^T\int_G |C(s,e)|^2\lambda(de)ds)^2\Big]\Big\}^{\frac{1}{2}}\Big)  \\
  &\leq \varepsilon\rho(\varepsilon) (\varepsilon\rho(\varepsilon)+\varepsilon)=\varepsilon^2\rho(\varepsilon).
\end{aligned}
$$
Let us calculate ii). From (\ref{equ 3.32})-ii), (\ref{equ 4.3}), (\ref{equ 3.4}) and the boundness of  $f_{\mu_1\mu_2}$, it
is easy to check
$$
\begin{aligned}
&\mathbb E\Big[\Big(\int_0^T\widetilde{\mathbb E}\Big[ \widetilde{X}^{1,\varepsilon}(s)
\widehat{\mathbb{E}}[f_{\mu_1\mu_2}(s;\widehat{\widetilde{\Lambda}}^*(s))(\widehat{\breve{P}}(s)+\widehat{M}(s))]\Big]ds\Big)^2 \Big]\\
&\leq \Big\{\mathbb E\bigg[\sup_{t\in[0,T]}|X^{1,\varepsilon}(s)|^4] \Big\}^{\frac{1}{2}}
\Big\{\mathbb E\widetilde{\mathbb E}\Big[\int_0^T|\widehat{\mathbb{E}}[f_{\mu_1\mu_2}(s;\widehat{\widetilde{\Lambda}}^*(s))(\widehat{\breve{P}}(s)+\widehat{M}(s))]|^4\Big] \Big\}^{\frac{1}{2}}\\
&\leq \varepsilon^2\rho(\varepsilon).
\end{aligned}
$$
According to the Assumption (A3.3),  (\ref{equ 3.4}), (\ref{equ 3.23}), (\ref{equ 3.32})  and (\ref{equ 4.3}) and the above two estimates,
applying H\"{o}lder inequality again, it yields
$\mathbb E\bigg[\bigg(\int_0^TI_{5,2}(s)ds\bigg)^2\bigg]\leq   \varepsilon^2\rho(\varepsilon).$
Besides, thanks to the continuous property of the second-order derivatives of $f$, one can
check the validity of $\mathbb E\bigg[\bigg(\int_0^TI_{5,3}(s)ds\bigg)^2\bigg]\leq   \varepsilon^2\rho(\varepsilon)$ easily.
Hence, we prove $\mathbb E\bigg[\bigg(\int_0^TI_5(s)ds\bigg)^2\bigg]\leq \varepsilon^2\rho(\varepsilon). $

Finally,  following the above argument, it also yields that
\begin{equation}\label{equ 6.29}
\begin{aligned}
&\bullet\ \int_0^1\widetilde{\mathbb E}[(f^\theta_{\mu_2}(s;\widetilde{\Lambda}^*(s))-f_{\mu_2}(s;\widetilde{\Lambda}^*(s)))
(\widetilde{\breve{P}}(s)+\widetilde{M}(s))]d\theta\\
&\qquad\qquad\qquad\qquad\qquad\qquad\qquad\qquad=\frac{1}{2}\widetilde{\mathbb E}[f_{\mu_2 a_2}(s;\widetilde{\Lambda}^*(s))(\widetilde{Y}^1(s))^2(\widetilde{X}^{1,\varepsilon}(s))^2]+I_6(s),
\end{aligned}
\end{equation}
here $I_6(s)$ satisfying $\mathbb E\bigg[\bigg(\int_0^TI_6(s)ds\bigg)^2\bigg]\leq \varepsilon^2\rho^*(\varepsilon).$ \\
Combining (\ref{equ 6.13}), (\ref{equ 6.15-1}),\ (\ref{equ 6.27}),\ (\ref{equ 6.28}), (\ref{equ 6.29}), we have  (\ref{equ 4.16}).


\renewcommand{\refname}{\large References}{\normalsize \ }

\end{document}